\documentclass[11pt,a4paper,twoside]{article}
\usepackage{amsmath,amsfonts,amssymb,amsthm,bbm,latexsym,mathrsfs}
\usepackage{graphicx,color,epsfig,fancyhdr,dsfont}
\usepackage{enumerate}
\usepackage{hyperref}
\usepackage{indentfirst}
\usepackage[all]{hypcap}
\usepackage[affil-it]{authblk}
\allowdisplaybreaks
\newtheorem{thm}{Theorem}[section]
\newtheorem{lem}[thm]{Lemma}

\newtheorem{prop}[thm]{Proposition}
\newtheorem{defi}[thm]{Definition}
\newtheorem{rmk}[thm]{Remark}
\newtheorem{exmp}[thm]{Example}

\renewcommand{\theequation}{\arabic{section}.\arabic{equation}}
 \numberwithin{equation}{section}

\newcommand{\E}{\mathbb{E}}
\newcommand{\N}{\mathbb{N}}
\newcommand{\D}{\mathcal{D}}
\newcommand{\R}{\mathbb{R}}
\newcommand{\B}{\mathcal{B}}
\newcommand{\T}{\mathcal{T}}
\newcommand{\F}{\mathcal{F}}
\begin{document}

\title
{Ergodicity of Sublinear Markovian Semigroups
}

\author[1]{
Chunrong Feng}
\author[1,2]{Huaizhong Zhao}
\affil[1]{Department of Mathematical Sciences, Durham
University, Durham DH1 3LE, UK}
\affil[2]{Research Centre for Mathematics and Interdisciplinary Sciences, Shandong University, Qingdao, 266237, China}
\affil[ ]{chunrong.feng@durham.ac.uk, huaizhong.zhao@durham.ac.uk}
\date{}

\maketitle
\newcounter{bean}
\vskip-10pt

\begin{abstract}
In this paper, we study the ergodicity of invariant sublinear expectation of sublinear Markovian semigroup. For this, we first develop an ergodic theory of an expectation-preserving map on a sublinear expectation space. Ergodicity is defined as any invariant set either has $0$ capacity itself or its complement has $0$ capacity. We prove, under a general sublinear expectation space setting, the equivalent relation between ergodicity and the corresponding transformation operator having simple eigenvalue $1$, and also with Birkhoff type strong law of large numbers if the sublinear expectation is regular. For sublinear Markov process, we prove that its ergodicity is equivalent to the Markovian semigroup having eigenvalue $1$ and it is simple in the space of bounded measurable functions. As an example we show that $G$-Brownian motion $\{B_t\}_{t\geq 0}$ on the unit circle has an invariant expectation and is ergodic if and only if ${\mathbb E}(-(B_1)^2)<0$. Moreover, it is also proved in this case that  the invariant expectation is regular and the canonical stationary process has no mean-uncertainty under the invariant expectation. 
\medskip

\medskip

\end{abstract}

\pagestyle{fancy}
\fancyhf{}
\fancyhead[LE,RO]{\thepage}
\fancyhead[LO]{\small{Ergodicity of Sublinear Markov Processes}}
\fancyhead[RE]{\small{C. R. Feng and  H. Z. Zhao}}
\renewcommand{\headrulewidth}{0.1pt}

 \renewcommand{\theequation}{\arabic{section}.\arabic{equation}}

\section{Introduction}

 Let $(  \Omega,   {\cal F},  P)$ be a probability space. The measure theoretical ergodic theory deals with a
 measure preserving map $  \theta: (  \Omega,   {\cal F})\to (  \Omega,   {\cal F})$ such that
 $  \theta  P=  P$.
Recall that the measurable dynamical system $\{{  \theta}^n\}_{n\in {\mathbb N}}$
on $(  \Omega,   {\cal F},  P)$ is called ergodic if any invariant set $A\in   {\cal F}$, i.e.
${  \theta}^{-1}A=A$, has either full measure or zero measure.
Ergodicity describes the indecomposable property of the system (c.f. Walters \cite{walters}). 
The well known result of Birkhoff's theorem says that 
a dynamical system is ergodic if and only if
in the long run, the time average of a function along its trajectory is the same as the spatial average on the entire space
with respect to the stationary measure (Birkhoff \cite{birkhoff}, von Neumann \cite{von neumann1, von neumann2}).

Due to the spreading nature of random forcing, ergodicity is an important common feature of stochastic systems.
It has aroused enormous interests of mathematicians
(c.f. Da Prato-Zabczyk \cite{da-prato}, Durrett \cite{durrett}, Feng-Zhao \cite{feng}).
For a Markovian random dynamical system,
it is well known that $1$, is a simple 
eigenvalue of the Markovian semigroup iff the stochastic system is ergodic, and is a unique eigenvalue on the unit circle and is simple iff
the stochastic system is weakly mixing. The latter statement is equivalent to the Koopman-von Neumann theorem.  
Recently, the ergodic theory
for periodic measures was obtained, in which it was proved that the Markovian semigroup has eigenvalues, $\{{\rm e}^{i{2m\pi\over \tau} t}\}_{m\in {\mathbb Z}}$, for a $\tau>0$, on the unit circle apart from the eigenvalue $1$ (Feng-Zhao \cite{feng}). Moreover, invariant measures of quasi-periodic 
stochastic systems were observed in Feng-Qu-Zhao \cite{fqz2}.

On a completely different topic, the concept of sublinear expectation is central in probability and statistics under uncertainty,
measures of risk and superhedging in finance
(Artzner-Delbaen-Eber-Heath \cite{delbaen1}, Chen-Epstein \cite{chen1}, El Karoui-Peng-Quenez \cite{peng1}, Follmer-Shied \cite{follmer1}). 
A coherent risk measure, that is defined as a real valued (monetary value) functional with properties
of constant preserving (cash invariance), monotonicity, convexity and positive homogeneity, is equivalent to a sublinear expectation. A systematic stochastic analysis of nonlinear/sublinear expectation and G-Brownian motion has been given in the substantial work Peng \cite{peng2005, peng2006, peng2010}.  

It is clear now that the corresponding partial differential equations of G-diffusions are fully nonlinear parabolic partial differential equations. 
They give the Markovian semigroup of $G$-diffusion processes (Peng \cite{peng2005, peng2006, peng2010}). It is noted that fully nonlinear PDEs have been intensively studied in literature e.g.  in Cafarelli-Cabre \cite{cafarelli}, Krylov \cite{krylov1,krylov}, Lions \cite{lions}. More recently,  the viscosity solution of path dependent fully nonlinear PDEs has been of great interests (Ekren-Touzi-Zhang \cite{zhangj1, zhangj2}, Peng \cite{peng2011}). However, study of the dynamical properties of long time behaviour of G-diffusion processes is still missing.
In this context, an ergodic theory under the sublinear expectation setting will be key to this study.
Our results will give the invariant properties, equilibrium and the statistical property of the stochastic systems under uncertainty. 

It is worth noting that economists already observed ``nonlinearities" in the behaviour of real world trading in financial market due 
to heterogeneity of expectation-formation processes (Culter-Poterba-Summers \cite{cutler}, De Long-Shleifer-Summers-Waldmann \cite{delong}, Frankel-Froot \cite{frankel}, Greenwood-Shleifer \cite{greenwood}, Williams \cite{williams}). Potentially biased beliefs of future price movements drive the decision of stock-market participants and create ambiguous volatility. 
To use sublinear expectations and G-Brownian motions to model ambiguity has been attempted in mathematical finance
literature e.g. Chen-Epstein \cite{chen1}, Epstein-Ji \cite{epstein}.


In this paper, we will go beyond the measure space
framework  to study an ergodic theory in a nonlinear
functional setting.
The lack of the dominated 
convergence and the Riesz representation creates a lot of difficulty to the analysis.
But the topology of 
a sublinear expectation space is still rich enough for us to define the ergodicity.
We will establish its equivalence
with the indecomposable property and characterisation in terms of spectrum of transformation operators.
We will prove the law of large numbers (LLN) also implies ergodicity, but the converse holds under a 
regularity assumption. 
This set-up is a natural framework for the ergodicity of invariant expectation of continuous 
time sublinear Markovian semigroup
such as that of G-diffusions.

It is noted here that the convergence of the LLN we study in this paper is in the pathwise sense quasi-surely. Convergence of the LLN in the sense of distribution was obtained by Peng \cite{peng2010} for independently identically distributed (i.i.d.) random variable sequences. 
But an i.i.d. random variable sequence may not be stationary in non-additive probability setting (Feng-Wu-Zhao \cite{feng-wu-zhao}). On the other hand, in our theory the independence assumption is not needed. Thus the LLN in the ergodic sense we study here and Peng's LLN for i.i.d. random variable sequences have different conclusions under different assumptions. 

We study Markovian stochastic dynamical systems with noise over a sublinear expectation
space.
A canonical sublinear expectation space with an expectation preserving map is constructed
 from an invariant expectation by the nonlinear Kolmogorov extension theorem onto the lifted path space. There is a natural expectation preserving dynamical system on the canonical sublinear expectation space.
The ergodicity of the stochastic system is then 
given by that of the canonical system. Its equivalence with a spectral property 
of the Markovian semigroup is also established.


As an example we show that the $G$-Brownian motion $B_t=\sqrt {t} \xi$ on the unit circle, 
where $\xi$ has normal distribution $N(0, [\underline \sigma ^2, \bar\sigma ^2])$, has an ergodic invariant expectation if and only if $\underline \sigma ^2>0$. Moreover, the 
invariant expectation and its extension on the canonical path space are regular so a 
Birkhoff type law of large numbers holds.
It is also noted that the canonical stationary process, which is the process corresponding 
to the large time behaviour, has no mean-uncertainty under the invariant 
expectation. 

This paper is the first paper to study the ergodic theory on a sublinear expectation space. This study is very general to include both discrete time and continuous time cases. 
Extending ideas of this paper on
the discrete time case, 
ergodicity for capacity especially upper probability has been obtained in Feng-Wu-Zhao \cite{feng-wu-zhao}.  Inspired by this work, the ergodicity of upper expectations generated from 
periodic measures has also been obtained (Feng-Qu-Zhao \cite{feng1}).


 \section{Dynamical systems on sublinear expectation spaces and ergodicity}\label{sec2}

  We first briefly recall the concept of sublinear expectation for convenience.
  Let $( \Omega,  {\cal F})$
  be a measurable space, $ \D$ be the linear space of all $ {\cal F}$-measurable real-valued
  functions. In particular, the indicator functions of any $ {\cal F}$-measurable sets which will be used in this paper are included in $\D$.

 \begin{defi} \label{zhao522}(c.f. Peng \cite{peng2010})
 A sublinear expectation $ \E$ is a functional $ \E:  \D\to \R$ satisfying\\
 (i) Monotonicity: $$ \E[X]\geq  \E[Y],\ {\rm if}\  X\geq Y.$$
 (ii) Constant preserving: $$ \E[c]=c,\ {\rm for}\ c\in \R.$$
 (iii) Sub-additivity: for each $X, Y\in  \D$,
 $$ \E[X+Y]\leq  \E[X]+ \E[Y].$$
 (iv) Positive homogeneity:
 $$ \E[\lambda X]=\lambda  \E[X],\ {\rm for}\ \lambda \geq 0.$$
 The triple $( \Omega,  \D, \E)$ is called a sublinear expectation space. 
 \end{defi}

The ergodicity concept in the sublinear situation is very subtle due to the absence of the linearity for functionals. 
The essence of the ergodicity is indecomposibility of dynamical systems. 
However, unlike in the classical ergodic theory, a set $A$ satisfying 
$ \E{\rm I}_{A}=1$ does not imply $ \E{\rm I}_{A^c}=0$ as the sublinear expectation $ \E$
only satisfies
\begin{eqnarray}
 \E{\rm I}_A+ \E{\rm I}_{A^c}\geq 1.
\end{eqnarray}
In fact it is quite possible that $ \E{\rm I}_{A}=1$ and $ \E{\rm I}_{A^c}=1$.  As a consequence, it is not viable to extend the classical definition of ergodicity which says that
any invariant set $A$ has either probability $0$ or $1$ to
$ \E I_A=0$ or $1$ in the sublinear case. 

The nonadditivity also creates a lot of technical difficulty to the analysis of its dynamics due to the missing of many important analysis tools such as the dominated 
convergence and the Riesz representation. 
But the topology of 
a sublinear expectation space is still rich enough for us to define the ergodicity in which the indecomposability is the most important property to survive. This is in line with the classical definition in measure theoretical ergodic theory. 
We observe that
three different forms of ergodicity in terms of invariant sets, spectrum of transformation operators and strong law of large numbers are still equivalent under the sublinear expectation setting with slightly stronger functionals satisfying the regularity given below. Without assuming conditions (iii) and (iv) in Definition \ref{zhao522}, it is still not clear how to define ergodicity in line with indecomposability. 

 The representation result (Artzner-Delbaen-Eber-Heath \cite{delbaen1}, Delbaen \cite{delbaen2}, Follmer-Schied \cite{follmer})
 says that there exists a family of linear expectations $\{E_{\theta}: \theta\in \Theta\}$
 defined on $ \D$ such that
 $$ \E[X]=\sup_{\theta\in \Theta}E_{\theta }[X].$$ 
 If it is assumed further that 
 \begin{eqnarray}\label{new1}
 \E[X_i]\to 0,
{\rm \ for\  each\ sequence\ of \ measurable\ functions}{\rm \ such\ that}\ 
 X_i(\omega)\downarrow 0 {\rm \ for\ each} \ \omega, 
  \end{eqnarray} 
 by Daniell-Stone theorem, the following {\it representation  as upper integrals} holds: there exists a family of $\sigma$-additive probability measures ${\cal P}$ 
 on $( \Omega,  {\cal F})$ such that
 \begin{eqnarray}\label{zhao521}
  \E[X]=\sup_{P\in {\cal P}}E_{P}[X].
 \end{eqnarray}
The representation as upper integrals is not essential to proceed our theory. We only need this in the proof of the LLN from the ergodicity.
We introduce the regularity of the following form:

\begin{defi}\label{zhao41}
 The functional $ \E[\cdot]$ is said to be regular if for any $A_n\in {\cal F}$, $A_n\downarrow \emptyset$, we have  $ \E[I_{A_n}]\downarrow 0$.
 \end{defi}


\begin{rmk}
(i). Definition \ref{zhao41} is equivalent to that if for any $A_n\in {\cal F}$, $A_n\downarrow A$ and $ \E I_A=0$ we have 
$ \E[I_{A_n}]\downarrow 0$. This can be see from
\begin{eqnarray*}
| \E[I_{A_n}]- \E[I_{A}]|\leq  \E[I_{A_n\setminus A}].
\end{eqnarray*}

(ii) A similar condition in Definition \ref{zhao41} in the capacity setting was also used in literature e.g. Cerreia-Vioglio-Maccheroni-Marinacci \cite{cerreia} where this was called continuous.

(iii) In Lemma  \ref {lem3.6} and Proposition \ref{prop3.26}, we will prove that the semigroup and the invariant expectation for $G$-Brownian motion on $S^1$ are regular.

(iv) The main results of this section are the relationships of ergodicity (E), the simpleness of eigenvalue 1 of the transformation operator $U_t$ on the bounded and measurable function space  (SE) and the law of large numbers (LLN). We prove (E) $\iff$ (SE) $\Leftarrow$ (LLN) without the regular condition.  It is needed only in the proof of (E) $\Rightarrow$ (LLN).
\end{rmk}

 Now we introduce a measurable transformation
$ \theta: \Omega\to \Omega$ that preserves the sublinear expectation $ \E$, i.e.
\begin{eqnarray}
 \theta \E= \E.
\end{eqnarray}
Here $ \theta \E$ is defined as
\begin{eqnarray*}
 \theta \E[X(\cdot)]= \E[X( \theta\cdot)] \ \ {\rm for\ any}  \  X\in  \D.
\end{eqnarray*}
Set the transformation operator $U_1:  \D\to  \D$ by
$$U_1\xi( \omega)=\xi( \theta \omega),\ \xi\in  {\cal D}.
$$
Then expectation  preserving of $ \theta$ is equivalent to
$$
 \E[U_1\xi]= \E[\xi], \  {\ \rm for \ any } \ \xi\in  \D.
$$

Define ${ \theta}^n= \theta\circ \theta\circ\dots\circ \theta$, $n\in {\mathbb N}$. Then $\{
{ \theta}^n\}_{n\in \mathbb N}$ forms a family of measurable transformations from $( \Omega,  {\cal F})$
to itself and satisfies expectation preserving property and the semigroup property:
\begin{eqnarray}
{ \theta}^{m+n}={ \theta}^{m}\circ { \theta}^{n},  \ {\rm for } \ n,m\in {\mathbb N}.
\end{eqnarray}
Thus $\{
{ \theta}^n\}_{n\in \mathbb N}$  is a dynamical system on $( \Omega,  {\cal D},  \E)$ and preserves
the sublinear expectation. In the following we will denote $ S=( \Omega,  {\cal D},  \E,
\{{ \theta}^n\}_{n\in \mathbb N})$ the dynamical system.

We call that a statement holds quasi-surely if it is true for all $ \omega \in  \Omega \setminus A$ for a set $A$ with $\E [I_A]=0$
and $v$-almost surely ($v-a.s.$) if it is true for all $\omega\in  \Omega \setminus A$
for a set $A$ with $\E [-I_A]=0$.

If a set $B\in  {\cal F}$ satisfies
\begin{eqnarray}
{ \theta}^{-1}B=B,
\end{eqnarray}
then we say the set $B$ is invariant with respect to the transformation
$ \theta$. First we prove the following result. 

\begin{thm}\label{hz5}
If $ \theta:  \Omega\to  \Omega$ is a measurable expectation
preserving transformation of the sublinear expectation space $( \Omega,  \D, \E)$, 
then the statements, 

\item{(i)} Any invariant measurable set $B\in  {\cal F}$ with respect to $ \theta$
satisfies either $ \E{\rm I}_B=0$ or $ \E{\rm I}_{B^c}=0$;

\item{(ii)} If $B\in  {\cal F}$ and $ \E{\rm I}_{{ \theta}^{-1}B\bigtriangleup  B}=0$, then either $ \E{\rm I}_B=0$ or $ \E{\rm I}_{B^c}=0$;

\item{(iii)} For every $A\in  {\cal F}$ with $ \E{\rm I}_A>0$, we have $ \E{\rm I}_{(\bigcup \limits_{n=1}^{\infty}{ \theta}^{-n}A)^c}=0$;

\item{(iv)} For every $A, B\in  {\cal F}$ with $ \E{\rm I}_A>0$ and $ \E{\rm I}_B>0$,
there exists $n\in {\mathbb N}^+$ such that $ \E{\rm I}_{({ \theta}^{-n}A\cap B)}>0$,

\noindent
have the following relations:  (i) and (ii) are
equivalent; (iii) implies (iv); (iv) implies (i). Moreover, if $ \E$ is regular, then (ii) implies (iii) and all the above four 
statements are equivalent.
\end{thm}

In (ii) above, $\cdot\bigtriangleup\cdot $ means the symmetric difference. The proof of this theorem is postponed to the Appendix. The result is similar to that in the classical ergodic theory, 
but it needs to deal with the issues due to the lack of 
additivity of probability. It is crucial to establish relations of these four statements especially their equivalence when $ {\mathbb E}$ is regular under the sublinear expectation setting. We will see that the combination of sublinearity and Statement (i) enable us to establish the ergodic theory. 
  We now discuss Statement (i) more closely. 
Note if the set  $B$ is invariant, then it is easy to see that
${ \theta}^{-1}(B^c)=B^c$. Thus in the case that $0< \E{\rm I}_B\leq 1$ and
$0<  \E{\rm I}_{B^c}\leq 1$, we could study $ \theta$ by studying two simpler transformations $ \theta|_{B}$
and $ \theta|_{B^c}$ separately.  In contrary, if $ \E{\rm I}_B=0$ and
$ \E{\rm I}_{B^c}=1$, we only need to study  $ \theta|_{B^c}$. Similarly, if $ \E{\rm I}_B=1$ and
$ \E{\rm I}_{B^c}=0$, we only need to study  $ \theta|_{B}$.  In the latter two cases, the transformation is indecomposable.
However it is noted that $ \E{\rm I}_{B}=0$ implies  $ \E{\rm I}_{B^c}=1$
and $ \E{\rm I}_{B^c}=0$ implies  $ \E{\rm I}_{B}=1$.
With the above observations, we give the following definition.

\begin{defi}
Let $( \Omega,  {\cal D}, \E)$ be a sublinear expectation space. An expectation
preserving transformation $ \theta$ of $( \Omega, {\cal D},  \E)$ is called ergodic if any invariant measurable set $B\in  {\cal F}$
satisfies either $ \E{\rm I}_B=0$ or $ \E{\rm I}_{B^c}=0$.
\end{defi}
\begin{thm}\label{hz16}
If $( \Omega, \D,  \E)$ is a sublinear expectation space and the measurable map $ \theta:  \Omega\to  \Omega$ is expectation preserving, 
then the following statements are equivalent:

\item{(i)}. The map $ \theta$ is ergodic;

\item{(ii)}.
 Whenever $\xi:  \Omega\to {\mathbb R} \ (or \ \ {\mathbb C})$ is  measurable, bounded quasi-surely and
 $U_1\xi=\xi$, then $\xi$ is constant quasi-surely.

If $ \E$ is regular, then (i) and (ii) are equivalent to
 \item{(iii)}.
 Whenever $\xi:  \Omega\to {\mathbb R}\  (or\  {\mathbb C})$ is measurable and $U_1\xi=\xi$ quasi-surely, then $\xi$ is constant quasi-surely.

%
\end{thm}

\begin{proof} 
It is trivial to see that (iii)$\Rightarrow$(ii). 

(ii)$\Rightarrow$(i). 
Consider $A\in  {\cal F}$ as an invariant set. Note ${\rm I}_A$ is bounded measurable and
satisfies $U_1{\rm I}_A={\rm I}_A$ quasi-surely. Thus ${\rm I}_A$ is constant quasi-surely. So ${\rm I}_A=0$ or $1$.
If ${\rm I}_A=0$ quasi-surely, then $ \E{\rm I}_A=0$. If ${\rm I}_A=1$ quasi-surely, then ${\rm I}_{A^c}=1-{\rm I}_A=0$ quasi-surely, so $ \E{\rm I}_{A^c}=0.$
That is to say either $ \E{\rm I}_A=0$ or $ \E{\rm I}_{A^c}=0$. Thus $ \theta $ is ergodic.

(i)$\Rightarrow$(iii). Now we assume $ \E$ is regular.
Let $ \theta$ be ergodic, $\xi$ be measurable and $U_1\xi=\xi$ quasi-surely. We assume $\xi$ to be real-valued as if $\xi$ is 
complex-valued, we can consider the real and imaginary parts
separately.  we will prove $\xi$ is a constant.
For a number $\alpha\in \R$, define
$A_\alpha=\{ \omega: \xi( \omega)>\alpha\}$ and $A_\alpha^c=\{ \omega: \xi( \omega)\leq\alpha\}$.
Note that $ \xi( \theta  \omega)= \xi( \omega)$ quasi-surely and 
$( \theta^{-1}A_\alpha)\bigtriangleup  A_\alpha\subset \{ \omega:\xi( \theta \omega)\ne \xi( \omega)\}$, we have $ \E {\rm I}_{( \theta^{-1}A_\alpha)\bigtriangleup  A_\alpha}=0$. By assumption that $ \theta$ is ergodic and Theorem \ref{hz5}, we know that $
 \E [{\rm I}_{A_\alpha}]=0$ or $ \E [{\rm I}_{A_\alpha^c}]=0$. Thus $
 \E [{\rm I}_{A_\alpha}]=0$ or $1$. Let $J:=\{\alpha:  \E [{\rm I}_{A_\alpha}]=0\}$.  By regular property of $ \E$, we have 
$$0= \E [{\rm I}_{\{ \omega: \xi( \omega)=\infty\}}]= \E [{\rm I}_{\cap_{n=1}^\infty A_n}]=\lim_{n\to \infty}  \E [{\rm I}_{A_n}].$$
Thus there exists $n\in \mathbb N$ such that $ \E [{\rm I}_{A_n}]=0$, that is $n\in J$ which implies $J\neq \emptyset$. So set $\alpha_*=\inf J$ and immediately $\alpha_*\in J$ by monotone (increasing) convergence of sublinear expectation. Hence for any $\alpha>\alpha_*$, we have $ \E [{\rm I}_{A_\alpha}]=0$, and for any $\alpha<\alpha_*$, we have $ \E [{\rm I}_{A_\alpha}]=1$ and $ \E [{\rm I}_{A^c_{\alpha}}]=0$ by ergodicity. By monotone (increasing) convergence of sublinear expectation again, we have $ \E [{\rm I}_{\{ \omega:\ \xi( \omega)<\alpha_*\}}]=0$. Combining $ \E [{\rm I}_{\{ \omega:\ \xi( \omega)>\alpha_*\}}]=0$ and the subadditivity of $ \E$, we have $ \E [{\rm I}_{\{ \omega:\ \xi( \omega)\neq\alpha_*\}}]=0$. 
Thus $\xi$ is constant quasi-surely.

 From the proof (i)$\Rightarrow$(iii), we can see that the regular assumption is used to prove $J\neq \emptyset$. But if $\xi$ is bounded quasi-surely, this is true automatically. So we do not need the regular assumption to obtain the equivalence of (i) and (ii) . 
\end{proof}

We give the definition of the
strong law of large numbers (SLLN).

 \begin{defi}\label{zhao524}
A dynamical system $ S=\{ \Omega,  \F,  \E,({ \theta}^n)_{n\in {\mathbb N}}\}$ is said to satisfy the
strong law of large numbers (SLLN) if for any bounded measurable function $\xi$, there exists a constant $c$ such that 
\begin{eqnarray}\label{hz15}
\lim_{N\to \infty} {1\over N} \sum _{n=0}^{N-1}\xi ({ \theta}^n\omega)=c, \ {\rm quasi-surely}.
\end{eqnarray}
\end{defi}
\begin{rmk}
(i) The SLLN in general may have random limit. But here we are interested in the relationship between SLLN and ergodicity. So in this paper, we define SLLN in a strong sense as that with the limit being constant.

(ii) In fact, it will be shown that the ergodicity and the SLLN are equivalent if $ \E$ is regular.
Without the regularity assumption, the SLLN still implies ergodicity, but it is not clear whether the vice versa is true. Thus, unlike the classical case, SLLN may not be used as  the definition of the dynamical system $\{\theta^n\}_{n\in \N}$ being ergodic unless it is regular.

\end{rmk}

Let $L_b( \F)$ be the space of all $ {\cal F}$-measurable real-valued
functions such that 
$\sup_{ \omega\in \Omega} |X( \omega)|<\infty$. As $U_11=1$ by definition of $U_1$, so it is obvious that $1$ is an eigenvalue of $U_1: {L}_b\to {L}_b$. The following result is almost obvious, but fundamental.
\begin{thm}\label{hz10}
If $ S$ satisfies SLLN, then the eigenvalue $1$ of $U_1$ on $L_b$ is simple and $\hat\theta$ is ergodic.
\end{thm}
\begin{proof}
Consider $\xi$ that satisfies
$$U_1\xi=\xi$$
 and $\xi$ is a bounded measurable r.v. 
  Thus by the SLLN assumption, we have
 $\xi$ is constant quasi-surely. Therefore the eigenvalue $1$ of $U_1$ is simple. Finally by Theorem \ref{hz16},
  $ \theta$ is ergodic. 
\end{proof}

We now investigate the converse part of Theorem {\ref{hz10}. For this we study the Birkhoff's ergodic theorem under sublinear expectation. Before doing this, we need the following lemma. The expectation preserving property of $ \theta$ is not required in Lemma \ref{lem2.13} and Lemma \ref{zhao523}.

\begin{lem} (Maximal ergodic lemma)\label{lem2.13} Let $\xi\in L^1( \Omega)$, $\xi_j( \omega)=\xi( \theta_j \omega)$, and $S_0=0,$
\begin{eqnarray}
&&S_k( \omega)=\xi_0( \omega)+\dots+\xi_{k-1}( \omega), \ {\rm for}\ k\geq 1,\label{eqn2.21}\\
&&M_k( \omega)=\max_{0\leq j\leq k} S_j( \omega).\label{eqn2.22}
\end{eqnarray}
Then for $k\geq 1$,
$$ \E[\xi {\rm I}_{\{M_k( \omega)>0\}}]\geq 0.$$
\end{lem}

\begin{proof} The proof is similar to the case of linear expectation given by Garsia (1965), so omitted here.
\end{proof}

We call a random variable $\xi$ has no mean uncertainty under $ \E$ if $ \E[\xi]=- \E[-\xi]$. Define the space for some $p\geq 1$, $L^p(\Omega):=\{\xi\in {\cal D}: \E|X|^p<\infty\}$,
$${\cal H}^p:=\{\xi\in L^p ( \Omega) : \xi\ {\rm has\ no\ mean\ uncertainty}\},$$
and
$${\cal H}_{\mathbb C}^p:=\{\xi\in L_{\mathbb C}^p ( \Omega) : \xi\ {\rm has\ no\ mean\ uncertainty}\}.$$
We have the following lemma which will be used later. Note here we do not need the regularity assumption.
\begin{lem}\label{lem2.10}
The space ${\cal H}^p$ (and ${\cal H}_{\mathbb C}^p$) is a Banach space.
\end{lem}
\begin{proof}
First note ${\cal H}^p$ (${\cal H}_{\mathbb C}^2$) is a linear subspace of $L^p ( \Omega)$ ($L_{\mathbb C}^p ( \Omega)$). We only need to prove the real valued random variable case. To see this, assume $\xi_1, \xi_2\in L^p ( \Omega)$ satisfy
 $$ \E[\xi_1]=-\E[-\xi_1], \  \E[\xi_2]=- \E[-\xi_2],$$
then by the sublinearity of $\hat\E$
\begin{eqnarray*}
 \E[\xi_1+\xi_2]\leq \E[\xi_1]+ \E[\xi_2]=- \E[-\xi_1]- \E[-\xi_2]\leq - \E[-(\xi_1+\xi_2)].
\end{eqnarray*}
So
$$ \E[\xi_1+\xi_2]+ \E[-(\xi_1+\xi_2)]\leq 0.$$
But
$$ \E[\xi_1+\xi_2]+ \E[-(\xi_1+\xi_2)]\geq 0.$$
Therefore
$$ \E[\xi_1+\xi_2]+ \E[-(\xi_1+\xi_2)]=0,$$
i.e. $\xi_1+\xi_2$ has no mean-uncertainty. Since $\xi_2$ has no mean-uncertainty, so does $-\xi_2$. Thus from what we have proved, we conclude that $\xi_1-\xi_2$ has no mean-uncertainty.

Consider $\lambda_1, \lambda_2>0$. Note $ \E[\lambda_1\xi_1]=\lambda_1 \E[\xi_1]$ and $ \E[-\lambda_1\xi_1]=\lambda_1 \E[-\xi_1]$. Thus if $\xi_1$ has no mean-uncertainty, so does $\lambda_1\xi_1$. Similarly if $\xi_2$ has no mean-uncertainty, so does $\lambda_2\xi_2$. Then by what we have proved, $\lambda_1\xi_1+\lambda_2\xi_2$ has no mean-uncertainty. Now when $\lambda_1>0, \lambda_2<0$, if $\xi_1$ and $\xi_2$ have no mean-uncertainty, then $\lambda_1\xi_1$ and $-\lambda_2\xi_2$ have no mean-uncertainty. Hence $\lambda_2\xi_2$ has no mean-uncertainty. Thus $\lambda_1\xi_1+\lambda_2\xi_2$ have no mean-uncertainty. This claim is also true for $\lambda_1<0, \lambda_2>0$ and $\lambda_1, \lambda_2<0$. Therefore  $\lambda_1\xi_1+\lambda_2\xi_2\in {\cal H}^p$.

Assume $\xi_n\in {\cal H}^p$ is a Cauchy sequence and with the limit $\xi\in {L}^p( \Omega)$, i.e.
\begin{eqnarray}\label{4.1}
\lim_{n\to 0} \E|\xi-\xi_n|^p=0.
\end{eqnarray}
 Then let us show that $\xi$ also has no mean-uncertainty. In fact,
\begin{eqnarray*}
 \E[\xi]&\leq&  \E[\xi-\xi_n]+ \E[\xi_n]\\
&=&  \E[\xi-\xi_n]- \E[-\xi_n]\\
&\leq &  \E[\xi-\xi_n]+ \E[-\xi+\xi_n]- \E[-\xi].
\end{eqnarray*}
Then let $n\to \infty$, we know the first two terms in above will go to $0$ because of (\ref{4.1}). Thus $ \E[\xi]\leq - \E[-\xi]$. But $ \E[\xi]\geq - \E[-\xi]$, so $ \E[\xi]= - \E[-\xi]$, i.e. $\xi$ has no mean-uncertainty so that $\xi\in {\cal H}^p$.
\end{proof}

The following result is an extension of the Birkhoff ergodic theorem to the case of sublinear expectation with the regularity assumption and the representation as upper integrals. Let ${\cal I}\subset  {\cal F}$ be the collection of such sets $A$ such that $\E{\rm I}_{( \theta^{-1}A)\bigtriangleup A}=0$. It is easy to check that ${\cal I}$ is a $\sigma$-field and $X\in {\cal I}$ iff $X(\theta\omega)=X(\omega)$ quasi-surely. Therefore, for any $\xi\in L^1( \Omega )$ and each $P\in {\cal P}$, as $E_P[\xi |{\cal I}]$ is ${\cal I}$-measurable, so $E_P[\xi |{\cal I}] ( \omega)=E_P[\xi |{\cal I}] ( \theta  \omega)$ quasi-surely. Define $\bar \xi^*, \underline \xi^* $ to be ${\cal I}$-measurable random variables such that 
\begin{eqnarray*}
\underline \xi^*\leq E_P[\xi |{\cal I}] \leq \bar \xi^*,
\end{eqnarray*}
quasi-surely for each $P\in {\cal P}$.  The proof of the following lemma is given in the Appendix. 
 
\begin{lem}\label{zhao523}
Assume $ \E$ is regular and has the representation as upper integrals. Then for any $\xi\in L_b$ and $\epsilon>0$,
\begin{eqnarray}\label{hz11}
\bar \xi ( \omega ):=\limsup _{n\to \infty}{1\over n} \sum_{m=0}^{n-1} \xi({ \theta}^m \omega)\leq \bar \xi^*( \omega )
+\epsilon,\ v-a.s.,
\end{eqnarray}
and
\begin{eqnarray}\label{hz13}
{\underline \xi}( \omega):=\liminf _{n\to \infty}{1\over n} \sum_{m=0}^{n-1} \xi({ \theta}^m \omega)\geq 
\underline \xi^*( \omega)-\epsilon ,\ v-a.s.
\end{eqnarray}
and ${\bar \xi }( \omega)$ and $\underline \xi( \omega)$ satisfy ${\bar \xi }( \theta \omega)={\bar \xi }( \omega)$ and ${\underline \xi }( \theta \omega)={\underline \xi }( \omega)$.
\end{lem}

 We use the notation of capacities from Chen \cite{chen} and Denis-Hu-Peng \cite{denis}. Under the assumption of the representation as upper integrals, we define a pair
$({\mathbb V}, v)$ of capacities by
$${\mathbb V}(A):=\sup_{P\in {\cal P}}P(A)=\E [I_ A],\ {v}(A) :=\inf_{P\in {\cal P}}P(A)=-\E [-I_A],\ {\rm for\ any}\ A\in  {\cal F},$$
which are called upper probability and lower probability. 
  A set function $\mu: {\cal  F}\to [0,1]$ is called continuous if $\lim_{n\to \infty} \mu (A_n)=\mu(A)$ when either $A_n\uparrow A$ or $A_n\downarrow A$ (Cerreia-Vioglio-Maccheroni-Marinacci \cite{cerreia}).

We need the following lemma from Cerreia-Vioglio-Maccheroni-Marinacci \cite{cerreia}. 

\begin{lem}\label{pr7}
Let $v$ be a continuous lower probability on $(\Omega,\mathcal{F})$. If $v$ is  $ \theta$-invariant, then for any bounded $\mathcal{F}$-measurable random variable $\xi$,
$$v\left(\left\{\omega:\ \lim_{n\to\infty}{1\over n}\sum_{k=0}^{n-1}\xi( \theta^{k}(\omega)) \hbox{ exists }\right\}\right)=1.$$
\end{lem}
This tells us that if $\E$ is regular and $ \theta$-invariant, we have 
$\lim\limits_{n\to\infty}{1\over n}\sum_{k=0}^{n-1}\xi( \theta^{k}(\omega))$ exists quasi-surely for any bounded measurable random variable $\xi$.

\begin{thm}\label{thm4.6} Assume $ \E$ is regular, $ \theta$-invariant and has the representation as upper integrals. If the dynamical system $ S$ is ergodic, then
SLLN holds and the constant in (\ref{hz15}) satisfies $c\in [- \E(-\xi), \E(\xi)]$.
\end{thm}
\begin{proof}
As $ \theta$ is $ \E$ preserving, so it is $v$-preserving. Moreover, $ \E$ is regular, by Lemma \ref{pr7}, we know that $\bar\xi ( \omega ):=\lim\limits_{n\to \infty}{1\over n} \sum_{m=0}^{n-1} \xi({ \theta}^m \omega)$ exists quasi-surely for any bounded measurable r.v. $\xi$ and ${\bar \xi }( \omega)$ satisfies ${\bar \xi }( \theta \omega)={\bar \xi }( \omega)$. As the dynamical system $ S$ is ergodic, so $\bar \xi=c$ is a constant. The SLLN is proved. 

On the other hand, we for any $P\in {\cal P}$, $E_P[\xi|{\cal I}]$ is ${\cal I}$-measurable, so $E_P[\xi |{\cal I}] ( \omega)=E_P[\xi |{\cal I}] ( \theta  \omega)$ quasi-surely. As  $ S$ is ergodic, by Theorem \ref{hz16}, $E_P[\xi |{\cal I}] $ is constant quasi-surely. Thus for any $P\in {\cal P}$ and any bounded measurable r.v. $\xi$,
\begin{eqnarray*}
E_P[\xi|{\cal I}]=E_P(\xi)\leq  \E(\xi),
\end{eqnarray*}
 and 
 \begin{eqnarray*}
-E_P[-\xi|{\cal I}]=-E_P(-\xi)\geq - \E(-\xi), 
\end{eqnarray*}
quasi-surely. Thus we can take $\bar \xi^*= \E[\xi]$ and $\underline \xi^*=- \E[-\xi]$, by Lemma \ref{zhao523},  
\begin{eqnarray*}
- \E [-\xi]-\epsilon 
\leq  \bar \xi ( \omega ):=\lim _{n\to \infty}{1\over n} \sum_{m=0}^{n-1} \xi({ \theta}^m \omega)\leq  \E [\xi] +\epsilon,\  v-a.s.
\end{eqnarray*}
As $\bar\xi=c$ is a constant,  so it is obvious that $c\in [- \E(-\xi),\E(\xi)]$.
\end{proof}
\begin{rmk}
An SLLN with $c\in \left[\int_{\Omega}\xi \mathrm{d}v,\int_{\Omega}\xi \mathrm{d}V\right]$ in the case of an upper and lower probability setup was obtained in  Feng-Wu-Zhao \cite{feng-wu-zhao}, where 
$\int_{\Omega}\xi \mathrm{d}v,\int_{\Omega}\xi \mathrm{d}V$ are Choquet integrals.
It is noted that the bound obtained in Theorem \ref{thm4.6} is  better as $ [- \E(-\xi), \E(\xi)]\subset \left[\int_{\Omega}\xi \mathrm{d}v,\int_{\Omega}\xi \mathrm{d}V\right]$.
This can be easily seen due to the well-known fact that from definition of Choquet integral
\begin{eqnarray*}
\int_{\Omega}\xi \mathrm{d}V&=&\int _0^{\infty}V(\omega: \xi(\omega)\geq t)dt+\int _{-\infty}^0\bigg(V(\omega: \xi(\omega)\geq t)-1\bigg)dt\\
&=&\int _0^{\infty}\sup_{P\in {\cal P}}P(\omega: \xi(\omega)\geq t)dt+\int _{-\infty}^0\bigg(\sup_{P\in {\cal P}}P(\omega: \xi(\omega)\geq t)-1\bigg)dt\\
&\geq&\sup_{P\in {\cal P}}\int _0^{\infty}P(\omega: \xi(\omega)\geq t)dt+\sup_{P\in {\cal P}}\int _{-\infty}^0\bigg(P(\omega: \xi(\omega)\geq t)-1\bigg)dt\\
&\geq&\sup_{P\in {\cal P}}\left[\int _0^{\infty}P(\omega: \xi(\omega)\geq t)dt+\int _{-\infty}^0\bigg(P(\omega: \xi(\omega)\geq t)-1\bigg)dt\right]\\
&=&\sup_{P\in {\cal P}}\int _{\Omega }\xi dP\\
&=&  \E(\xi).
\end{eqnarray*}
Similarly one can prove that
\begin{eqnarray*}
\int_{\Omega}\xi \mathrm{d}v
&\leq & - \E(-\xi).
\end{eqnarray*}
\end{rmk}
\section{Sublinear Markovian systems and their ergodicity: the general setting}\label{zhao1751}

Consider a measurable space $(\Omega,{\cal F})$ with a similar notation such as ${\cal D}=L_b({\cal F})$ as in Section 2.
Let $(\Omega, {\cal D}, \E)$ a sublinear expectation space where $\E[\cdot]$ is a sublinear expectation on $L_b(\cal F)$. Denote by $C_{b,lip}(\R^d)$ be 
the space of real-valued bounded Lipschitz continuous functions on $\R^d$, $C_b(\R^d)$ the space of real-valued bounded continuous functions on $\R^d$. We denote by $L_b(\B(\R^d))$, the space of $\B(\R^d)$-measurable real-valued functions defined on $\R^d$ such that
$\sup_{x\in \R^d}|\varphi (x)|<\infty.$ Let $\xi\in (L_b({\cal F}))^{\otimes d}$ be given. The sublinear distribution of $\xi$ under $\E[\cdot]$ is defined by
$$T[\varphi]:=\E[\varphi(\xi)], \ \varphi\in L_b(\B(\R^d)).$$
This distribution $T[\cdot]$ is again a sublinear expectation defined on $L_b(\B(\R^d))$.
Denote by $S(d)$ the collection of symmetric $d\times d$ matrices and $S_+(d)$ the collection of positive definite symmetric $d\times d$ matrices.

Consider a family of sublinear expectations parameterized by $t\in \R^+$:
$$T_t: L_b(\B(\R^d))\to L_b(\B(\R^d)), \ {\rm t\geq 0}.$$
\begin{defi} (Peng \cite{peng2005})
The operator $T_t$ is called a sublinear Markov semigroup if it satisfies\\
(m1) For each fixed $(t,x)\in \R^+\times \R^d$, $T_t[\varphi](x)$ is a sublinear expectation defined on  $L_b(\B(\R^d))$;\\
(m2) $T_0[\varphi](x)=\varphi(x)$;\\
(m3) $T_t[\varphi](x)$ satisfies the following Chapman semigroup formula
$$(T_t\circ T_s)[\varphi]=T_{t+s}[\varphi],\ t,s\geq 0.$$
\end{defi}

There are many examples of sublinear Markov semigroups. We list some of them here,
though they were already
known, for the completeness and an aid to understand
the problem we address here.

\begin{exmp} (Lions \cite{lions}) Consider the Hamilton-Jacobi-Bellman equation:
\begin{eqnarray}
\begin{cases} {\partial\over \partial t}u&=\sup\limits _{v\in V}\{\sum\limits_{i,j=1}^d a_{ij}(x,v){\partial ^2\over \partial x_i\partial x_j}u+\sum _{i=1}^d
b_i(x,v){\partial\over \partial x_i}u\},
\cr
u(0,\cdot)&=\varphi(\cdot)\in C_b({\mathbb R}^d).
\end{cases}
\end{eqnarray}
Here $a: {\mathbb R}^d\times {\mathbb R}^k\to S(d)$ and  $b: {\mathbb R}^d\times {\mathbb R}^k\to {\mathbb R}^{d}$ are bounded and uniformly continuous functions,   and uniformly Lipschitz in $x$, $V$ is a closed and bounded subset of ${\mathbb R}^k$. Under the notion of viscosity solutions, this equation has a unique
solution $u(t,x)$ in $C_b({\mathbb R}^d)$ with initial value $\varphi$. Set
\begin{eqnarray*}
({T}_t\varphi)(x):=u(t,x), \ \ x\in {\mathbb R}^d.
\end{eqnarray*}
This defines a sublinear Markov semigroup.
\end{exmp}

\begin{exmp}(Peng \cite{peng2010})\label{example3.3}
Let $G:S(d)\to {\mathbb R}$ be a given sublinear  function which is monotonic on $S(d)$. Then there exists a bounded, convex and closed subset $\sum\subset S_+(d)$ such that
\begin{eqnarray*}
G(A)=\sup \limits _{B\in \sum }[{1\over 2}{\rm tr}(AB)], \ {\rm for } \ A\in S(d).
    \end{eqnarray*}
    Define $\Omega=C_0(\R^+, \R^d)$, the space of all $\R^d$-valued continuous functions $(\omega_t)_{t\in \R^+}$, with $\omega_0=0$, equipped with the distance
$$\rho (\omega^{1},\omega^{2}):=\sum_ {i=1}^\infty 2^{-i} [\max_{t\in [0,i]}|\omega_t^{1}-\omega_t^{2}|\wedge 1]$$
with ${\cal F}=\B (C_0(\R^+,\R^d))$. Let
$$L_{ip}({\Omega}):= \{\varphi(\omega_{t_1},\omega_{t_2},\dots, \omega_{t_m}), {\rm\ for \ any}\ m\geq 1, \ t_1, t_2,\dots,t_m\in \R^+, \varphi\in C_{b,Lip}((\R^d)^m) \}.$$
Then 
the $G$-normal distribution 
$N(\{0\}\times \sum)$
on $(\Omega,L_{ip}(\Omega))$ exists, i.e. there exists a $d$-dimensional random vector $X$ on a sublinear expectation space $(\Omega, {\cal D},\E)$ satisfying 
\begin{eqnarray*}
aX+b\bar X
=^{\hskip-7pt d}\sqrt{a^2+b^2}X, \ {\rm for}\ a, b\geq 0,
\end{eqnarray*}
where $\bar X$ is an independent copy of $X$ and $G(A)=\E [{1\over 2} \langle AX, X\rangle]$.
 It was proved in Theorem 2.5 in Chapter VI in 
Peng \cite{peng2010} that there exists a weakly compact family of probability measures ${\cal P}$ on $(\Omega,{\cal B}(\Omega))$ such that 
\begin{eqnarray*}
\E[X]=\max_{P\in \cal P}E_{P}[X], {\rm  \ for \ } X\in L_{ip}(\Omega).
\end{eqnarray*}
Its canonical path is $G$-Brownian motion $\{B_t\}_{t\geq 0}$ on a sublinear expectation space $(\Omega, {\cal D},\E)$ satisfying 

   (i). $B_0(\omega )=0$;

   (ii). For each $t,s\geq 0$, the increment $B_{t+s}-B_t$ is $N(\{0\}\times s\sum)$ distributed and independent of
   $(B_{t_1},B_{t_2},\dots,B_{t_n})$, for each $n\in {\mathbb N}$ and $0\leq t_1\leq t_2\leq\dots\leq t_n\leq t$.

   For each fixed $\varphi\in C_{b,Lip}({\mathbb R}^d)$, the function
   \begin{eqnarray}
   u(t,x):=\E\varphi(x+B_t), \ (t,x)\in [0,\infty)\times {\mathbb R}^d,
   \end{eqnarray}
   is the viscosity solution of the following $G$-heat equation
   \begin{eqnarray}
   {\partial \over \partial t}u=G(D^2u), \ u(0,\cdot)=\varphi (\cdot).
   \end{eqnarray}
   Then $({T}_t\varphi )(x)=u(t,x)$ defines a semilinear Markovian semigroup.
\end{exmp}

\begin{exmp}(Peng \cite{peng2010}) Let $\{B_t\}_{t\geq 0}$ be a $k$-dimensional $G$-Brownian motion on the sublinear expectation
space  $(\Omega,{\cal D},\E)$, $b:{\mathbb R}^d\to {\mathbb R}^{d}$, $\sigma:{\mathbb R}^d\to {\mathbb R}^{d\times k}$, $h:{\mathbb R}^d\to {\mathbb R}^{d\times k\times k}$ be global Lipschitz functions. Here $G: S(d)\to {\mathbb R}$ is a given sublinear function which is monotonic on $S(d)$.
Consider the stochastic differential equations on ${\mathbb R}^d$ driven by the $G$-Brownian motion $B$
\begin{eqnarray}
dX_t=b(X_t)dt+\sum _{i,j=1}^kh_{ij}(X_t)d\langle B^i,B^j\rangle_t+\sum _{i=1}^k\sigma_j(X_s)dB_t^j,
\end{eqnarray}
with initial condition $X_t=x$. Here $\langle \cdot,\cdot\rangle_{\cdot}$ is the mutual variation process. Define $F: S(d)\times \R^d \times \R^d\to S(d)$ with
\begin{eqnarray}
F_{ij}(A,p,x)={1\over 2}\langle A\sigma_i(x),\sigma_j(x)\rangle+\langle p,h_{ij}(x)+h_{ji}(x)\rangle.
\end{eqnarray}
Then ${T}_t\varphi(x)=\E\varphi(X_t)=:u(t,x)$ satisfies
\begin{eqnarray}
{\partial\over\partial t}u=G(F(D^2 u,Du,x))+bDu
\end{eqnarray}
and defines a sublinear Markovian semigroup for $\varphi\in C_{b, lip} (\R^d)$.
\end{exmp}

In this section, we will give the construction of canonical dynamical system on path space
under the assumption of the existence of invariant sublinear expectations of Markovian semigroups. 
Then we follow the standard philosophy in literature to define the ergodicity of the canonical dynamical system as the ergodicity of the stochastic dynamical systems (c.f. Da Prato-Zabczyk \cite{da-prato}).  
The invariant sublinear expectation has not been studied very much in literature. As far as we know, so far there is only one work (Hu-Li-Wang-Zheng \cite{hu})
on the existence of invariant sublinear expectation for G-diffusion processes if the system is sufficiently dissipative.

%
Firstly, we give the definition of an invariant expectation of sublinear Markovian semigroups as a natural extension of invariant measures.
\begin{defi}
An invariant sublinear expectation
$\tilde {T}: L_b({\B(\R^d}))\to \R$ is a sublinear expectation satisfying
$$(\tilde T T_s)(\varphi)=\tilde T(\varphi),\ {\rm for\  any}\ \varphi\in L_b(\B(\R^d)),$$
where $T_s, s\geq 0$ is a sublinear Markov semigroup.
\end{defi}

Define $\Omega^*=C(\R, \R^d)$, the space of all $\R^d$-valued continuous functions $(\omega^*_t)_{t\in \R}$ equipped with the distance
\begin{eqnarray}\label{eqn3.7a}
\rho (\omega^{*1},\omega^{*2}):=\sum_ {i=1}^\infty 2^{-i} [\max_{t\in [-i,i]}|\omega_t^{*1}-\omega_t^{*2}|\wedge 1]
\end{eqnarray}
with ${\cal F}^*=\B (C(\R,\R^d))$.  Moreover, set $ \hat \Omega =(\R^d)^{(-\infty,+\infty)}$ as the space of all $\R^d$-valued 
functions on $(-\infty,+\infty)$, $  \hat{\cal F}={\cal B}(\hat \Omega)$ is the smallest $\sigma$-field containing all cylindrical sets of $\hat \Omega$, $\hat {\cal D}$ is the linear space of all 
$  \hat{\cal F}$-measurable real-valued functions.

Given a sublinear Markov semigroup $T_t, t\geq 0$ and the invariant sublinear expectation $\tilde {T}[\cdot]$, we can define the family of finite-dimensional sublinear distributions of the canonical process
$( \omega_t)_{t\in \R}\in   \Omega$ under a sublinear expectation $\E^{\tilde {T}}[\cdot]$ on $((\R^d)^m,L_b(\B[(\R^d)^m]))$ as follows. For each integer $m\geq 1$, $\varphi\in L_b(\B[(\R^d)^m])$ and $t_1< t_2< \dots<t_m$, we successively define functions $\varphi_i\in L_b( \B[(\R^d)^{(m-i)}])$, $i=1,\dots, m,$ by
\begin{eqnarray*}
\varphi_1(x_1,\dots,x_{m-1})&:=&T_{t_m-t_{m-1}}[\varphi (x_1, \dots, x_{m-1},\cdot)](x_{m-1}),\\
\varphi_2(x_1,\dots,x_{m-2})&:=&T_{t_{m-1}-t_{m-2}}[\varphi_1 (x_1, \dots, x_{m-2},\cdot)](x_{m-2}),\\
&\vdots&\\
\varphi_{m-1}(x_1)&:=&T_{t_2-t_{1}}[\varphi_{m-2} (x_1, \cdot)](x_{1}).
\end{eqnarray*}

We now consider two different set-ups. The first one is to consider
 $\varphi_{m}:=\tilde T[\varphi_{m-1} (\cdot)]$ and
$$\E^{\tilde {T}}[\varphi( \hat \omega_{t_1}, \hat  \omega_{t_2},\dots,   \hat\omega_{t_m})]:=T^{\tilde T}_{t_1, t_2,\dots,t_m}[\varphi (\cdot)]:=\varphi_m.$$
In fact, $T_t^{\tilde T}=\tilde T$, for $t\geq 0$ and $T^{\tilde T}_{t_1, t_2,\dots,t_m}[\varphi (\cdot)]$ is a sublinear expectation defined on $L_b(\B[(\R^d)^m])$.
Denote
$$\tilde {\cal E}(\varphi(\hat \omega_0))=\tilde T [\varphi], {\rm \ for \ any \ } \varphi\in L_b(\B(\R^d)),$$
then
$$\tilde {\cal E}(\varphi(\hat \omega_t))=\tilde {\cal E}(\varphi(\hat \omega_0))=\tilde T [\varphi], {\rm \ for \ any \ } \varphi\in L_b(\B(\R^d)).$$

For an ordered set of distinct real numbers ${\mathbb I}=\{t_1,t_2,\dots,t_m\}$, let ${\mathbb I}^{\prime}=\{t_{\pi_1},t_{\pi_2},\dots,t_{\pi_m}\}$ be  a permutation of ${\mathbb I}$ so that $t_{\pi_1}<t_{\pi_2},\dots<t_{\pi_m}$. Define
\begin{eqnarray*}
{T}_{t_{1},t_{2},\dots,t_{m}}^{\tilde {T}}\varphi (x_1,x_2,\dots,x_m)={T}_{t_{\pi_1},t_{\pi_2},\dots,t_{\pi_m}}^{\tilde {T}}\varphi(x_{\pi_1},x_{\pi_2},\dots,x_{\pi_m}).
\end{eqnarray*}
The second set-up is to set $\varphi_m(x):=T_{t_1}[\varphi_{m-1} (\cdot)](x)$ for $t_1\geq 0$ following Peng \cite{peng2005}. Then 
$$\E^{x}[\varphi(\hat \omega_{t_1}, \hat \omega_{t_2},\dots, \hat \omega_{t_m})]:=T^{x}_{t_1, t_2,\dots,t_m}[\varphi (\cdot)]:=\varphi_m(x),$$
and $T^{x}_{t_1, t_2,\dots,t_m}[\cdot]$ defines a sublinear expectation. 

Set
$$L_0(\hat {\cal F}):= \{\varphi(\hat \omega_{t_1}, \hat \omega_{t_2},\dots, \hat \omega_{t_m}), {\rm\ for \ any}\ m\geq 1, \ t_1, t_2,\dots,t_m\in \R, \varphi\in L_b(\B[(\R^d)^m]) \}.$$
It is clear that $L_0(\hat {\cal F})$ is a linear subspace of $L_b(\hat {\cal F})$. Denote $L_0^p(\hat \Omega)$
that is the completion of $L_0(\hat {\cal F})$ under the norm $(\E^{\tilde {T}}[|\cdot|^p])^{1\over p}$, $p\geq 1$.
Define the space
\begin{eqnarray*}
&&Lip_{b,cyl}({\hat \Omega})\\
&:=& \{\varphi(\hat\omega_{t_1}, \hat\omega_{t_2},\dots, \hat\omega_{t_m}), {\rm\ for \ any}\ m\geq 1, \ t_1, t_2,\dots,t_m\in \R, \varphi\in C_{b,Lip}((\R^d)^m) \},
\end{eqnarray*}
and $L_G^p(\hat \Omega)$ the completion of $ Lip_{b,cyl}({\hat \Omega})$ under the norm $||\cdot||_{L_G^p}=(\E^{\tilde {T}}[|\cdot|^p])^{1 \over p}$. From Denis-Hu-Peng \cite{denis}, we know that the completion of $C_b(\hat \Omega)$ and $Lip_{b,cyl}({\hat \Omega})$ under the norm $||\cdot||_{L_G^p}$ are the same, and $L_G^p(\hat \Omega)\subset L_0^p(\hat \Omega)$. Here $C_b(\hat \Omega)$ is defined in a similar way as $Lip_{b,cyl}({\hat \Omega})$, but  replacing $C_{b,Lip}((\R^d)^m) $ by $C_{b}((\R^d)^m)$.

It was already known that there exists a unique sublinear  expectation $\E^x$ with finite 
dimensional expectation $\E^x=T^x_{t_1,t_2,\dots,t_m}$, $m\in \N$,  by applying the
nonlinear Kolmogorov extension theorem (\cite{peng2005}). 
For our purpose, by applying Kolmogorov's theorem again, there exists a unique
sub-linear expectation $\E^{\tilde {T}}$ on $L_0^1(\hat \Omega)$
such that
$$\E^{\tilde {T}} [Y]=T^{\tilde T}_{t_1, t_2,\dots,t_m}[\varphi (\cdot)],$$
for any $m\geq 1$, $t_1, t_2,\dots,t_m\in \R$, $Y\in L_0(\hat {\cal F})$ with $Y(\hat\omega)= \varphi(\hat \omega_{t_1}, \hat \omega_{t_2},\dots, \hat \omega_{t_m})$, $\varphi\in L_b(\B[(\R^d)^m])$.

Following the idea in \cite{peng2005}, we can also define the conditional expectation. Define $\hat \Omega_t:=\{\hat\omega \in \hat\Omega: \hat\omega_s\equiv\hat \omega_t, {\rm\ for\ any}\ s\geq t\}$ and $\hat{\cal F}_t:= {\cal B}(\hat\Omega_t)$. Let $X\in L_0(\hat{\cal F})$ be given as 
$$X=\varphi(\hat\omega_{t_1},\dots, \hat\omega_{t_n}, \hat\omega_{t_{n+1}}, \dots, \hat\omega_{t_{n+m}}), \ t_1<\dots<t_n<\dots<t_{n+m},$$
where $\varphi\in L_b(\B[(\R^d)^{n+m}])$. Without loss of generality, we may assume $t_n=t$. The conditional expectation under $\hat{\cal F}_t$ denoted by $\E^{\tilde {T}}[\cdot |\hat{\cal F}_t] : L_0(\hat{\cal F})\to L_0(\hat{\cal F}_t)$ is defined by
\begin{eqnarray}\label{neweqn3.8}
\E^{\tilde {T}}[X |\hat{\cal F}_t]:= \Phi (\hat\omega_{t_1}, \hat\omega_{t_2},\dots, \hat\omega_{t_n}),
\end{eqnarray}
where $\Phi(x_1, \dots, x_n):= T_{t_{n+1}-t_n,\dots, t_{n+m}-t_n}^{x_n} [\varphi (x_1,\dots,x_n,\cdot)]$. Similar to Proposition 5.1 in \cite{peng2005}, this can also be extended to $L_0^p(\hat\Omega)$.

Now we write the canonical process and associated $\sigma$-field as 
\begin{eqnarray}\label{cf12}
\hat X_t(\hat \omega)=\hat \omega_t,\ \hat \omega\in \hat \Omega,\ t\in \R.
\end{eqnarray}
The process $\hat X_t$, $t\in\R$, is Markovian in the sense that for $h>0$
\begin{eqnarray}\label{cf13}
&&\E^{\tilde {T}} [\varphi(\hat X(t+h))|\hat\F_t]
=T_h^{\hat X(t)}\varphi.
\end{eqnarray}

Now we introduce a group of invertible measurable transformations
$$\hat\theta_t\hat \omega(s)=\hat \omega(t+s),\ t,s \in \R.$$
Then it is easy to see that for any $\varphi\in L_0^1(\hat \Omega)$,
$$\E^{\tilde{T}}[\varphi(\hat X)]=\E^{\tilde{T}}[\varphi(\hat\theta_t\hat X)], $$
i.e.
$$\hat \theta_t\E^{\tilde{T}}=\E^{\tilde{T}}.$$
Thus $\hat\theta_t$ is an expectation preserving (or distribution preserving) transformation. Thus $S^{\tilde{T}}=(\hat \Omega, \hat {\cal D}, (\hat\theta_t)_{t\in {\mathbb R}}, \E^{\tilde{T}})$ defines a dynamical system, called {\it canonical dynamical system} associated with $T_t, t\geq 0$ and $\tilde{T}$, $\hat\theta_t$ preserving the expectation $\E^{\tilde{T}}$ for any function $\varphi\in L_0^1(\hat\Omega)$. The group $\hat\theta_t, t\in\R$ induces a group of linear transformation $U_t, t\in\R$, either on the real space $L_0^2 (\hat \Omega, \hat {\cal D},\E^{\tilde {T}})$  or $L_{0, \mathbb C}^2 (\hat \Omega, \hat {\cal D},\E^{\tilde {T}})$, by formula
$$U_t\xi(\hat \omega)=\xi(\hat\theta_t\hat \omega),\ \xi\in L_0^2 (\hat \Omega)\ ({\rm or}\ L_{0, \mathbb C}^2 (\hat \Omega)),\ \hat \omega\in \hat \Omega,\ t\in\R.$$
\begin{defi}
A dynamical system $S^{\tilde {T}}=(\hat \Omega, \hat {\cal D}, \hat\theta_t,\E^{\tilde {T}})$ is said to be continuous if for any $\xi\in L_0^2 (\hat \Omega)\ ({\rm or}\ L_{0, \mathbb C}^2 (\hat \Omega))$,
$$\lim_{t\to 0} U_t\xi=\xi, \ in \ L_0^2 (\hat \Omega)\ ({\rm or}\ L_{0, \mathbb C}^2 (\hat \Omega)).$$
\end{defi}
Denote
$$B (x, \delta)=\{y\in \R^d: |y-x|<\delta\}.$$
\begin{defi}
A stochastic process $\hat X(t)$, $t\in \R$ on $(\hat \Omega, \hat {\cal D},\E^{\tilde {T}})$ is said to be stochastically continuous if for any $\delta >0$,
$$\lim_{t\downarrow s}\E^{\tilde{T}}[{\rm I}_{\{|\hat X(t)-\hat X(s)|\geq \delta\}}]=0.$$
\end{defi}
\begin{defi}
A sublinear Markov semigroup $T_t, t\geq 0$ is said to be stochastically continuous if
$$T_t(x,B^c(x,\delta)):=\E^x[{\rm I}_{B^c(x,\delta)} (\hat X_t)]\downarrow 0, \ {\rm as}\ t\to 0, \ {\rm for\ any}\ x\in \R^d,\ \delta>0.$$
\end{defi}

\begin{thm}
If a Markov semigroup $T_t, t>0$ is stochastically continuous, then
$$\lim_{t\to 0}T_t f(x)=f(x),\ {\rm for\ all}\  f\in C_b (\R^d),\ x\in \R^d.$$
\end{thm}
\begin{proof}
For any $f\in C_b (\R^d)$, let $\epsilon>0$, $\delta>0$ be such that
$$|f(x)-f(y)|<\epsilon,\ {\rm provided}\ |x-y|<\delta.$$ So
\begin{eqnarray*}
&&|T_t f(x)-f(x)|\\
&=&|\E^x [f(\hat X(t))]-\E^x [f(\hat X(0))]|\\
&\leq & \E^x|f(\hat X(t))-f(\hat X(0))|\\
&=& \E^x|(f(\hat X(t))-f(\hat X(0))) {\rm I}_{\{|\hat X(t)-\hat X(0)|<\delta\}}|+ \E^x|(f(\hat X(t))-f(\hat X(0))) {\rm I}_{\{|\hat X(t)-\hat X(0)|\geq \delta\}}|\\
&\leq &\epsilon+2||f||_\infty \E^x[{\rm I}_{\{|\hat X(t)-\hat X(0)|\geq \delta\}}].
\end{eqnarray*}
Since $T_t$ is stochastically continuous, we have
$\lim_{t\to 0}T_t f(x)=f(x)$.
\end{proof}
\begin{prop}\label{prop3.6}
Let $T_t, t\geq 0$ be a sublinear Markov semigroup and $\tilde{T}$ is the invariant expectation. If the corresponding canonical process $\hat X(t), t\in \R$ on $(\hat \Omega, \hat \D,\E^{\tilde {T}})$ is stochastically continuous, then the dynamical system $S^{\tilde {T}}$ is continuous, i.e.
\begin{eqnarray}\label{3.1}
\lim_{s\to t}U_s\xi=U_t \xi,\ \xi\in L_G^2(\hat \Omega).
\end{eqnarray}
\end{prop}
\begin{proof} First we check (\ref{3.1}) for all $\xi\in Lip_{b, cyl}(\hat \Omega)$, i.e. for all $\xi$ of the form
$$\xi=f(\hat\omega_{t_1}, \hat\omega_{t_2},\dots, \hat\omega_{t_m}),$$
where $f\in C_{b,Lip}(\B[(\R^d)^m]),\ t_1<t_2<\dots<t_m.$
Let $\epsilon>0$, $\delta>0$ be such that
$$|f(x_1,\dots, x_m)-f(y_1,\dots, y_m)|<\epsilon,\ {\rm provided}\ |x_i-y_i|<\delta,\ i=1,\dots,m.$$
Then
\begin{eqnarray*}
&&\E^{\tilde {T}} |U_t\xi-U_s\xi|^2\\
&=&\E^{\tilde {T}} |f(\hat\omega(t_1+t),\dots, \hat\omega(t_m+t))-f(\hat\omega(t_1+s),\dots, \hat\omega(t_m+s))|^2\\
&=&\E^{\tilde {T}} |f(\hat X(t_1+t),\dots, \hat X(t_m+t))-f(\hat X(t_1+s),\dots, \hat X(t_m+s))|^2\\
&\leq &\E^{\tilde {T}}\Big[ |f(\hat X(t_1+t),\dots, \hat X(t_m+t))-f(\hat X(t_1+s),\dots, \hat X(t_m+s))|^2\\
&&\hskip 1cm  {\rm I}_{\{|\hat X({t_i+t})-\hat X({t_i+s})|<\delta,\ {\rm for\ any}\ i=1,\dots, m\}}\Big] \\
&&+\E^{\tilde {T}}\Big[ |f(\hat X(t_1+t),\dots, \hat X(t_m+t))-f(\hat X(t_1+s),\dots, \hat X(t_m+s))|^2\\
&&\hskip 1cm  {\rm I}_{\{|\hat X({t_i+t})-\hat X({t_i+s})|\geq\delta,\ {\rm for\ some}\ i=1,\dots, m\}}\Big]\\
&\leq& \epsilon +2||f||_\infty^2 \sum_{i=1}^m\E^{\tilde {T}}\Big[ {\rm I}_{\{|\hat X({t_i+t})-\hat X({t_i+s})|\geq\delta\}}\Big].
\end{eqnarray*}
Since $\hat X_t$ is stochastically continuous, 
(\ref{3.1}) follows for all $\xi\in Lip_{b, cyl}(\hat \Omega)$.

For any $\xi\in L_G^2(\hat \Omega)$, there exist $\xi_n\in Lip_{b,cyl}({\hat \Omega})$ such that for any $\epsilon>0$, there exists $N>0$, such that for any $n\geq N$, we have
$$\E^{\tilde {T}}|\xi_n-\xi|^2<{\epsilon \over 9}.$$
Now for the fixed $N$, there exists a $\delta>0$,
$$\E^{\tilde {T}} |U_t\xi_N-U_s\xi_N|^2<{\epsilon\over 9},\ {\rm when}\ |t-s|<\delta.$$
Therefore
\begin{eqnarray*}
\E^{\tilde {T}} |U_t\xi-U_s\xi|^2
&\leq& 3\left[\E^{\tilde {T}} |U_t\xi-U_t\xi_N|^2+\E^{\tilde {T}} |U_t\xi_N-U_s\xi_N|^2+\E^{\tilde {T}} |U_s\xi_N-U_s\xi|^2\right]\\
&\leq& 3\left [\E^{\tilde {T}} |\xi-\xi_N|^2+\E^{\tilde {T}} |U_t\xi_N-U_s\xi_N|^2+\E^{\tilde {T}} |\xi_N-\xi|^2\right ]\\
&<&\epsilon.
\end{eqnarray*}
The proposition is proved.
\end{proof}
\begin{rmk}
When we discuss the ergodicity of $G$-Brownian motion on $S^1$, we can show that $\hat X$ has a continuous modification which is also stochastically continuous in Proposition \ref{prop3.26}.
\end{rmk}
\begin{prop}\label{thm3.6}
Let $T_t, t\geq 0$ be a stochastically continuous Markov semigroup and $\tilde{\cal E}$ satisfy (\ref{new1}). Then the corresponding canonical process $\hat X(t), t\in \R$ on $(\hat \Omega, \hat \D,\E^{\tilde {T}})$ is stochastically continuous.
\end{prop}
\begin{proof}
Assume that $T_t, t\geq 0$ is stochastically continuous, then for any $t>s$ and $\delta>0$, we have
\begin{eqnarray*}
\E^{\tilde {T}}[{\rm I}_{\{|\hat X(t)-\hat X(s)|\geq \delta\}}]
  &=&\E^{\tilde {T}}\Big[\E^{\tilde {T}}[{\rm I}_{\{|\hat X(t)-\hat X(s)|\geq \delta\}}|{\cal F}_s]\Big]\\
  &=&\E^{\tilde {T}}[T_{t-s}(\hat X(s), B^c(\hat X(s),\delta))]\\
   &=&{\tilde{\cal E}}[T_{t-s}(\hat X(s), B^c(\hat X(s),\delta))],
 \end{eqnarray*}
by Markov property. Note here the conditional expectation can be defined in the Markovian case as we already explained in (\ref{neweqn3.8}) (Peng \cite{peng2005}). Since $T_t, t\geq 0$ is stochastically continuous and $\tilde{\cal E}$ satisfies (\ref{new1}), we have
$$\lim_{t\downarrow s}\E^{\tilde {T}}[{\rm I}_{\{|\hat X(t)-\hat X(s)|\geq \delta\}}]=0.$$
\end{proof}


Mirrored by the discrete case discussed in Section \ref{sec2}, we can give the following definitions.

\begin{defi}
A set $A\in {\hat \F}$ is said to be invariant with respect to $S^{\tilde {T}}=(\hat \Omega, \hat \D, \hat\theta_t,\E^{\tilde {T}})$
if for any $t\in \R$, $\hat\theta_t^{-1}A=A$.
\end{defi}


\begin{defi}
The invariant expectation $\tilde {T}$ is said to be ergodic with respect to the Markov semigroup ${T}_t, t\geq 0$, if its associated canonical
dynamical system $S^{\tilde {T}}=(\hat \Omega, \hat {\cal D}, \hat\theta_t, \E^{\tilde {T}})$ is ergodic i.e. any invariant set $A$ satisfies either $\E^{\tilde {T}} [{\rm I}_A]=0$ or $\E^{\tilde {T}} [{\rm I}_{A^c}]=0$.
 \end{defi}

As $U_t 1=1$ by definition of $U_t$. So it is obvious that $1$ is an eigenvalue of $U_t: L_b(\hat{\cal F})\to L_b(\hat{\cal F})$. Similar to the proof 
of Theorem \ref{hz16} we can prove:

\begin{thm}\label {thm3.19}
The continuous dynamical system $S^{\tilde {T}}$ is ergodic if and only if the eigenvalue $1$ of $U_t$ on $L_b(\hat{\cal F})$ is simple.
\end{thm}

\begin{defi}
A dynamical system $S^{\tilde {T}}=(\hat \Omega, \hat {\cal D}, (\hat\theta_t)_{t\in {\mathbb R}} , \E^{\tilde {T}})$ is said to satisfy the strong law of large numbers (SLLN) if for any bounded measurable function $\xi$, there exists a constant $c$ such that
\begin{eqnarray}\label{hz17}
\lim_{T\to \infty} {1\over T} \int_0^T U_t\xi dt=c, \ {\rm quasi-surely}.
\end{eqnarray}
\end{defi}

\begin{thm}\label{thm4.3}
If $S^{{\tilde T}}$ satisfies SLLN, then the eigenvalue $1$ of $U_t$ on $ L_b(\hat{\cal F})$ is simple and $S^{{ \tilde T}}$
is ergodic.
\end{thm}
\begin{proof} The proof is similar to that of Theorem \ref{hz10}.
\end{proof}

Now let us prove the converse part of Theorem \ref{thm4.3} under the regularity assumption.
\begin{thm}\label{thm4.7}
Assume the eigenvalue $1$ of $U_t$ on $L_b(\hat{\cal F})$ is simple and $\E^{\tilde {T}}$ is regular. Then the dynamical system $S^{\tilde {T}}$ satisfies SLLN and the constant in (\ref {hz17}) satisfies \\
$c\in[-\E^{\tilde {T}} (-\int_0^1 U_t\xi dt), \E^{\tilde {T}} (\int_0^1 U_t\xi dt)]$.
\end{thm}
\begin{proof}
Assume $1$ is a simple eigenvalue of $U_t$ on $L_b(\hat{\cal F})$. For an arbitrary $h>0$, $\xi\in  L_b(\hat{\cal F})$, $\xi\geq 0$, define
$$\xi_h=\int_0^h U_s\xi ds,$$
and consider $\hat \theta_h$, a fixed expectation preserving transformation on $\hat \Omega$. Then
$${1\over n} \sum_{k=0}^{n-1} \xi_h(\hat\theta_h^{k}(\hat \omega))={1\over n}\int_0^{nh}U_s\xi(\hat \omega)ds,$$
and $\E^{\tilde {T}}$ is regular, by Theorem \ref{thm4.6},
$$-\E^{\tilde {T}}[-\xi_h]\leq \lim_{n\to \infty} {1\over n}\int_0^{nh}U_s\xi ds=:\bar \xi_h^*\leq \E^{\tilde {T}}[\xi_h], \ \rm quasi-surely.$$
For arbitrary $T\geq 0$, let $n_T=[{T\over h}]$ be the maximal nonnegtive integer less than or equal to ${T\over h}$. Then  $n_T h\leq T\leq (n_T+1)h$ and quasi-surely
$${{n_T}\over {(n_T+1)}h}{1\over {n_T}}\int_0^{n_Th}U_s\xi ds\leq {1\over T}\int_0^{T}U_s\xi ds\leq{{n_T+1}\over {n_T}h}{1\over {n_T+1}}\int_0^{(n_T+1)h}U_s\xi ds.$$
Thus,
$$\lim_{T\to \infty}{1\over T}\int_0^{T}U_s\xi ds= {1\over h}\bar \xi_h^*, \ \rm quasi-surely.$$
In particular, it follows that $\bar \xi_h^*=h\bar\xi_1^*$. But it is easy to see that
$$U_h\bar\xi_h^*=\bar\xi_h^*.$$
Thus
$$U_h\bar\xi_1^*=\bar\xi_1^*,\ {\rm for \ all}\ h\geq 0.$$
However, from the assumption, $\bar\xi_1^*$ should be a constant quasi-surely. So
$$-\E^{\tilde {T}}[-\int_0^1 U_t\xi dt]=-\E^{\tilde {T}}[-\bar\xi_1^*]\leq \bar\xi_1^*= \E^{\tilde {T}}[\bar\xi_1^*]\leq \E^{\tilde {T}}[\xi_1]=\E^{\tilde {T}}[\int_0^1 U_t\xi dt].$$
This proves that the dynamical system $S^{\tilde {T}}$ satisfies the SLLN.
\end{proof}

%
\begin{prop}\label{lem5.1}
If $\varphi \in L_b({\cal B}(\R^d))$ satisfies 
$T_t\varphi=\varphi$, $T_t(-\varphi)=-\varphi$ and $|\varphi (\hat \omega(0))|^2$ has no mean-uncertainty, 
then $\xi\in L_0^2$ given by
$$\xi(\hat \omega)=\varphi(\hat \omega(0)),\ \hat \omega\in \hat \Omega,$$
satisfies $U_t\xi=\xi$, quasi-surely.
\end{prop}
\begin{proof}
Note
$$U_t\xi(\hat \omega)=\xi(\hat \theta_t\hat \omega)=\varphi(\hat\theta_t\hat \omega(0))=\varphi(\hat \omega(t)).$$
So the condition that $U_t\xi= \xi,$ quasi-surely, is equivalent to
$$\varphi(\hat \omega(t))= \varphi(\hat \omega(0)),\ {\rm quasi-surely}$$
and therefore
\begin{eqnarray}\label{5.1}
\varphi(\hat X(t))=\varphi(\hat X(0)),\ {\rm quasi-surely},
\end{eqnarray}
where $\hat X(t),\ t\in \R$ is the canonical process. To prove (\ref {5.1}), note that
\begin{eqnarray*}
&&\E^{\tilde {T}} |\varphi(\hat X(t))-\varphi(\hat X(0))|^2\\
&\leq& 2\E^{\tilde {T}} \big[-\varphi (\hat X(t))\varphi (\hat X(0))\big]+\E^{\tilde {T}} |\varphi(\hat X(t))|^2+\E^{\tilde {T}} |\varphi(\hat X(0))|^2.
\end{eqnarray*}
By Markovian property and the assumption
that $T_t\varphi=\varphi$, $T_t(-\varphi)=-\varphi$ and $|\varphi (\hat \omega(0))|^2$ has no mean-uncertainty, we have
\begin{eqnarray*}
 &&\E^{\tilde {T}} \Big[-\varphi (\hat X(t))\varphi (\hat X(0))\Big]\\
 &=&\E^{\tilde {T}} \bigg[\E^{\tilde {T}}\Big[-\varphi (\hat X(t))\varphi (\hat X(0))|\hat {\cal F}_0\Big]\bigg]\\
 &\leq&\E^{\tilde {T}} \bigg[\big(-\varphi (\hat X(0))\big)^+\E^{\tilde {T}}\Big[{\varphi (\hat X(t))}|\hat {\cal F}_0\Big]+\big(-\varphi (\hat X(0))\big)^-\E^{\tilde {T}}\Big[{-\varphi (\hat X(t))}|\hat {\cal F}_0\Big]\bigg]\\
 &=&\E^{\tilde {T}} \Big[\big(-\varphi (\hat X(0))\big)^+ (T_t\varphi )(\hat X(0))+\big(-\varphi (\hat X(0))\big)^- 
 (T_t\big(-\varphi ))(\hat X(0)\big)\Big]\\
 &=&\E^{\tilde {T}} \Big[\big(-\varphi (\hat X(0))\big)^+ \varphi (\hat X(0))+\big(-\varphi (\hat X(0))\big)^- \big(-\varphi (\hat X(0))\big)\Big]\\
 &=&\E^{\tilde {T}} \big[-|\varphi (\hat X(0))|^2\big]\\
 &=&-\E^{\tilde {T}} |\varphi (\hat X(0))|^2.
\end{eqnarray*}
Note also
$$\E^{\tilde {T}}|\varphi(\hat X(t))|^2=\E^{\tilde {T}} |\varphi(\hat X(0))|^2.$$
So
\begin{eqnarray*}
\E^{\tilde {T}} |\varphi(\hat X(t))- \varphi(\hat X(0))|^2
\leq -2\E^{\tilde {T}} |\varphi (\hat X(0))|^2+2\E^{\tilde {T}} |\varphi (\hat X(0))|^2=0.
\end{eqnarray*}
Thus
$$\E^{\tilde {T}} |\varphi(\hat X(t))- \varphi(\hat X(0))|^2=0.$$
It follows that
$$|\varphi(\hat X(t))- \varphi(\hat X(0))|=0, \ {\rm quasi-surely}.$$
The result is proved.
\end{proof}
\begin{lem}\label{lem5.2}
Assume that $\xi\in L_0^2$ satisfies $U_t\xi= \xi$, quasi-surely. Then for an arbitrary random variable $\tilde\xi\in L_0^2$ which is $\hat {\cal F}_{[-t,t]}$-measurable, $t\geq 0$, we have
$$\E^{\tilde {T}}\Big|\E^{\tilde {T}}\big[ U_t\tilde\xi|\hat {\cal F}_{[0,0]}\big]-\xi\Big|^2\leq 10\E^{\tilde {T}}|\xi-\tilde \xi|^2. $$
\end{lem}
\begin{proof}
First we have for the sublinear expectation, for $t\geq 0$,
\begin{eqnarray*}
&&\E^{\tilde {T}}\Big|\E^{\tilde {T}}\big[ U_t\tilde\xi|\hat {\cal F}_{[0,0]}\big]-\xi\Big|^2\\
&\leq & 2\E^{\tilde {T}}\Big|\E^{\tilde {T}}\big[U_t\tilde\xi|\hat {\cal F}_{[0,0]}\big]-U_{-t}\tilde\xi\Big|^2+2\E^{\tilde {T}}| U_{-t}\tilde\xi-\xi|^2\\
&= & 2\E^{\tilde {T}}\Big|\E^{\tilde {T}}\big[ U_t\tilde\xi|\hat {\cal F}_{0}\big]-\E^{\tilde {T}}\big[ U_{-t}\tilde\xi|\hat {\cal F}_0\big]\Big|^2+2\E^{\tilde {T}}| U_{-t}\tilde\xi-U_{-t}\xi|^2\\
&=&2\E^{\tilde {T}}\Big|\E^{\tilde {T}}\big[ U_t\tilde\xi|\hat {\cal F}_{0}\big]-\E^{\tilde {T}}\big[ U_{-t}\tilde\xi|\hat {\cal F}_0\big]\Big|^2+2\E^{\tilde {T}}|\tilde\xi-\xi|^2,
\end{eqnarray*}
where we have used that $\hat X$ is a Markov process, that $U_{t}\tilde\xi$  and $U_{-t}\tilde\xi$ are respectively $\hat {\cal F}_{[0,2t]}$- and $\hat {\cal F}_{0}$- measurable and that $U_t$ is $\E^{\tilde {T}}$-preserving transformation.

By Jensen's inequality and sublinearity of $\E^{\tilde {T}}$, we have
\begin{eqnarray*}
\Big|\E^{\tilde {T}}\big[ U_t\tilde\xi|\hat {\cal F}_{0}\big]-\E^{\tilde {T}}\big[ U_{-t}\tilde\xi|\hat {\cal F}_0\big]\Big|^2
&\leq &\left|\E^{\tilde {T}}\Big[|U_t\tilde\xi- U_{-t}\tilde\xi|\big|\hat {\cal F}_0\Big]\right|^2\\
&\leq &\E^{\tilde {T}}\Big[|U_t\tilde\xi- U_{-t}\tilde\xi|^2\big|\hat {\cal F}_0\Big].
\end{eqnarray*}
Moreover, it follows from $\E^{\tilde {T}}$-preserving property of $U_t$ that
\begin{eqnarray*}
\E^{\tilde {T}}\bigg[\E^{\tilde {T}}\Big[|U_t\tilde\xi- U_{-t}\tilde\xi|^2\big|\hat {\cal F}_0\Big]\bigg]
&=&\E^{\tilde {T}}\Big[\big| U_t\tilde\xi- U_{-t}\tilde\xi\big|^2\Big]\\
&=&\E^{\tilde {T}}\Big[\big| U_{2t}\tilde\xi- \tilde\xi\big|^2\Big]\\
&\leq & 2\E^{\tilde {T}}\Big[\big| U_{2t}\tilde\xi- U_{2t}\xi\big|^2\Big]+2\E^{\tilde {T}}\Big[\big|U_{2t}\xi- \tilde\xi\big|^2\Big]\\
&=&2\E^{\tilde {T}}\Big[\big| \tilde\xi-\xi\big|^2\Big]+2\E^{\tilde {T}}\Big[\big|\xi-\tilde\xi\big|^2\Big]\\
&\leq &4\E^{\tilde {T}}|\tilde\xi-\xi|^2.
\end{eqnarray*}
The result follows.
\end{proof}
Now we are ready to prove the converse part of Proposition \ref{lem5.1}.
\begin{prop}\label{lem5.3}
If $\xi\in L_0^2(\hat\Omega)$ and $U_t\xi= \xi$, then there exists $\varphi \in L_b ({\cal B}(\R^d))$ such that $T_t\varphi=\varphi$, 
$T_t(-\varphi)=-\varphi$ and $\xi(\hat \omega)=\varphi (\hat \omega(0))$ quasi-surely.
\end{prop}
\begin{proof}
For $\xi\in L_0^2(\hat\Omega)$, by definition of $L_0^2(\hat\Omega)$, there exits a sequence $\{\tilde \xi_n\}$ of ${\cal F}_{[-nt,nt]}$-measurable elements of $L_b(\hat{\cal F})$ such that
$$\E^{\tilde {T}}|\tilde \xi_n-\xi|^2\to 0,\ {\rm as}\ n\to\infty.$$
Thus by Lemma \ref{lem5.2},
$$\lim_{n\to \infty}\E^{\tilde {T}}[ U_{nt}\tilde\xi_n|{\cal F}_{[0,0]}]=\xi \ {\rm in }\ L_0^2.$$
Moreover, there exists $\varphi_n\in L^2_{\mathbb C} (\R^d, \tilde T)$ such that
$$\E^{\tilde {T}}[U_{nt}\tilde\xi_n|{\cal F}_{[0,0]}]=\varphi_n (\hat X(0)), \ {\rm quasi-surely}.$$
Thus
$$\lim_{n\to \infty}\varphi_n (\hat X(0))=\xi,  \ {\rm in }\ L_0^2(\hat\Omega).$$
By Borel-Cantelli lemma (Denis-Hu-Peng \cite{denis}), we can choose a quasi-surely convergent subsequence, still denoted by $\varphi_n (\hat X(0))$. 
Now we define
\begin{eqnarray*}
\varphi(x)=
\left\{\begin{array}{ll}
\lim_{n\to \infty} \varphi_n(x)  \ &{\rm if\ the\ limit\ exists},\\
0\  &{\rm otherwise}.
\end{array}
\right.
\end{eqnarray*}
Then $\xi=\varphi (\hat X(0))$. It follows from $U_t\xi=\xi$ that 
\begin{eqnarray*}
\varphi (\hat X(t))=U_t\varphi (\hat X(0))=\varphi (\hat X(0)).
\end{eqnarray*}
By using conditional expectations, we have
\begin{eqnarray*}
(T_t\varphi )(\hat X(0))=\E^{\tilde {T}}[\varphi (\hat X(t))|{\cal F}_0]=\E^{\tilde {T}}[\varphi (\hat X(0))|{\cal F}_0]=\varphi (\hat X(0)),
\end{eqnarray*}
and
\begin{eqnarray*}
(T_t(-\varphi ))(\hat X(0))=\E^{\tilde {T}}[-\varphi (\hat X(t))|{\cal F}_0]=\E^{\tilde {T}}[-\varphi (\hat X(0))|{\cal F}_0]=-\varphi (\hat X(0)).
\end{eqnarray*}
The proof is complete.
\end{proof}

\begin{thm}\label{hz20}
Assume the Markov chain $T_t$ has an invariant  expectation $\tilde T$. Let $\hat X$ be the canonical processes on the canonical dynamical system $(\hat \Omega, \hat \D, \hat \theta_t,\E^{\tilde {T}})$ and stochastically continuous. 
Then  the following two statements:

(i)  if $T_t\varphi=\varphi$, $T_t(-\varphi)=-\varphi$, $\varphi\in L_b({\cal B}(\R^d))$ for any $t\geq 0$, then $\varphi$ is constant, $\tilde T$-quasi-surely;

(ii) $\tilde T$ is ergodic,\\
have the relation that (i) implies (ii). Moreover, if we assume further that for any $\varphi\in L_b({\cal B}(\R^d))$,   $|\varphi(\hat X(0))|^2$ has no mean-uncertainty, then (i) and (ii) are equivalent.
\end{thm}
\begin{proof}
The theorem can be proved easily by Theorem \ref{thm3.19}, Proposition \ref{lem5.1} and Proposition \ref{lem5.3}.
\end{proof}

\section{Ergodicity of G-Brownian motion on the unit circle}

As an example, we consider a $G$-Brownian motion on the unit circle $S^1=[0,2\pi]$ defined by $X(t)=x+B_t\ {\rm mod}\  2\pi$, where $B$ is a one-dimensional $G$-Brownian motion such that $B_1$ has normal distribution $N(0,[\underline \sigma^2,\overline \sigma ^2])$. Here
$\overline \sigma ^2   \geq \underline \sigma^2$ are constants. See Example \ref{example3.3} for the definition of the $G$-Brownian motions.  For $\varphi\in C_{b,lip}(S^1)$, set 
\begin{eqnarray}\label{eqn3.7}
T_t\varphi(x)=u(t,x)=\E \varphi(X(t)).
\end{eqnarray}
 Then $u$ is a viscosity solution of the following fully nonlinear PDE (Peng \cite{peng2006, peng2010})
\begin{eqnarray}\label{eqn3.8}
{\partial \over \partial t}u={1\over 2} \overline \sigma ^2u_{xx}^+-{1\over 2} \underline \sigma ^2u_{xx}^-, \ u|_{t=0}=\varphi,\ x\in S^1.
\end{eqnarray}
If we assume $\underline\sigma^2>0$, according to Krylov \cite{krylov1, krylov}, or Peng \cite{peng2010}, when $t>0$, $u(t,x)$ is $C^{1,2}$ in $(t,x)$, thus a classical solution for any $t>0$. In fact, we can extend the solution to the case when $\varphi$ is bounded and measurable and obtain a classical solution for any $t>0$. Before we give this result, we need the following lemma about the regularity of $T_t$.
\begin{lem}\label{lem3.6}
Assume $\underline\sigma^2>0$, for $T_t$ defined in (\ref{eqn3.7}), we have for any $t>0$, $A_n\in {\cal B}({S^1})$ such that $A_n\downarrow \emptyset$, we have $(T_t I_{A_n})(x) \downarrow 0$.
\end{lem}
\begin{proof} From Denis-Hu-Peng \cite {denis}, we know that for any function $\varphi\in L_b({\cal B}(S^1))$, 
\begin{eqnarray}\label{zhao45}
T_t\varphi(x)=\E \varphi(X(t))=\sup_{\theta^2_\cdot \in \{ {\rm adapted \ processes \ with \ values\  in }\ [\underline\sigma^2, 
\overline\sigma^2]\}} E[\varphi(x+\int_0^t \theta_s dW_s\ {\rm mod}\ 2\pi)],
\end{eqnarray}
where $W_.$ is the classical Brownian motion on $\R^1$, $W_0=0$,  and $E$ is the linear expectation with respect to $W_.$. Denote ${\cal F}_t=\sigma \{ W_s: 0\leq s\leq t\}$.
Note that by Theorem 3.4.6 in Karatzas-Shreve \cite{karatzas},  $\int_0^t \theta_s dW_s$ is in law a Brownian motion with time $\tilde \theta _t^2=\int _0^t\theta _s^2ds$ i.e. 
there exists a standard Brownian motion $\tilde W$ such that $\int_0^t \theta_s dW_s=\tilde W_{\tilde \theta _t^2}$,
where $\tilde\theta_t^2$ is a stopping time with respect to 
the filtration ${\cal G}_s={\cal F}_{\tau(s)}$, where $\tau(s)=\inf\{t\geq 0:\tilde \theta_t^2>s\}$. Note that 
$\underline \sigma ^2\leq  \theta_.^2\leq  \bar\sigma ^2$ and $\underline \sigma ^2>0$, so $\tilde \theta _t^2$ is strictly increasing in $t$, we have $\tau(s)=\inf\{t\geq 0:\tilde \theta_t^2\geq s\}$. Define $\tilde\tau:=\tau(\underline\sigma^2t)$. It is easy to see that $\tilde\tau\leq t$ and $\int_0^{\tilde\tau} \theta_s dW_s=\tilde W_{\tilde\theta_{\tilde\tau}^2}=\tilde W_{\underline \sigma^2t}$, which is a Brownian motion w.r.t. ${\cal G}_{\underline \sigma ^2t}={\cal F}_{\tilde\tau}$. Therefore
 \begin{eqnarray}\label{zhao46}
&&E[\varphi(x+\int_0^t \theta_s dW_s\ {\rm mod}\ 2\pi)]\nonumber\\
&=&E\left [E[\varphi(x+\int_0^{\tilde\tau} \theta_s dW_s+\int_{\tilde\tau}^t \theta_s dW_s\ {\rm mod}\ 2\pi)|{\cal F}_{\tilde\tau}]\right]\nonumber\\
&=&E\left [E[\varphi(x+y+\int_{\tilde\tau}^t \theta_s dW_s\ {\rm mod}\ 2\pi)]|_{y=\int_0^{\tilde\tau} \theta_s dW_s}\right]\nonumber\\
&=&E\left [E[\varphi(x+y+\int_{\tilde\tau}^t \theta_s dW_s\ {\rm mod}\ 2\pi)]|_{y=\tilde W_{\underline \sigma^2t}}\right]\nonumber\\
&=&\int_0^{2\pi} p(\underline \sigma^2 t, y) E[\varphi(x+y+\int_{\tilde\tau}^t \theta_s dW_s\ {\rm mod}\ 2\pi)] dy\nonumber\\
&=&E[\int_0^{2\pi} p(\underline \sigma^2 t, y) \varphi(x+y+\int_{\tilde\tau}^t \theta_s dW_s\ {\rm mod}\ 2\pi)dy]\nonumber\\
&=&E[\int_0^{2\pi} p(\underline \sigma^2 t, y) \varphi(x+y+z\ {\rm mod}\ 2\pi)dy|_{z=\int_{\tilde\tau}^t \theta_s dW_s}]\nonumber\\
&=& E\left[E[\varphi(x+z+\tilde W_{\underline \sigma ^2t}\ {\rm mod}\ 2\pi)]|_{z=\int_{\tilde\tau}^t \theta_s dW_s}\right],
\end{eqnarray}
where $p(\cdot, \cdot)$ is the heat kernel of Brownian motion $\tilde W_{\cdot}$ on $S^1$ starting at position $0$ at time $0$. In fact,
\begin{eqnarray*}
&&E[\varphi(x+z+\tilde W_{\underline \sigma ^2t}\ {\rm mod}\ 2\pi)]\\
&=& \int_{S^1} p(\underline \sigma ^2t, y-x-z\  {\rm mod}\ 2\pi) \varphi(y) dy\\
&=&
\sum\limits _{k\in {\mathbb Z}} \int_0^{2\pi}{1\over \sqrt{2\pi\underline \sigma ^2t}}{\rm e}^{-{(x+z\ {\rm mod}\ 2\pi-y-2k\pi)^2\over 2\underline \sigma ^2t}}\varphi (y) dy.
\end{eqnarray*}
 So for any $A_n\in {\cal B}({S^1})$, using inequality $(a-b)^2\geq {1\over 2}a^2-b^2$, we have 
\begin{eqnarray}\label{zhao44}
&&E[I_{A_n}(x+z+\tilde W_{\underline \sigma ^2t}\ {\rm mod}\ 2\pi)]\nonumber\\
&=& 
\sum\limits _{k\in {\mathbb Z}} \int_0^{2\pi}{1\over \sqrt{2\pi\underline \sigma ^2t}}{\rm e}^{-{(x+z\ {\rm mod}\ 2\pi-y-2k\pi)^2\over 2\underline \sigma ^2t}}I_{A_n} (y) dy\nonumber\\
&\leq&  \int_0^{2\pi}I_{A_n} (y){1\over \sqrt{2\pi\underline \sigma ^2t}}{\rm e}^{{(x+z\ {\rm mod}\ 2\pi-y)^2\over 2\underline \sigma ^2t}}\sum\limits _{k\in {\mathbb Z}}{\rm e}^{-{(2k\pi)^2\over 4\underline \sigma ^2t}} dy\nonumber\\
&\leq& {\rm Leb}(A_n) {1\over \sqrt{2\pi\underline\sigma^2 t}}{\rm e}^{{(2\pi)^2\over 2\underline\sigma^2 t}} {1\over {1-e^{-{\pi^2\over \underline \sigma ^2t
}}}}.
\end{eqnarray}
 Note the upper bound of (\ref{zhao44}) is independent of $x, z$ and $\theta _\cdot$, so it follows from 
(\ref{zhao45}) and (\ref{zhao46}) that
\begin{eqnarray*}
&&
(T_tI_{A_n})(x)\\
&=&\sup_{\theta^2_\cdot \in \{ {\rm adapted \ processes \ with \ values\  in }\ [\underline\sigma^2, 
\overline\sigma^2]\}}
E\left [
E[I_{A_n}(x+z+\tilde W_{\underline \sigma ^2t}\ {\rm mod}\ 2\pi)]|_{z=\int_\tau^t \theta_s dW_s}\right ]\\
&\leq& {\rm Leb}(A_n) {1\over \sqrt{2\pi\underline\sigma^2 t}}{\rm e}^{{(2\pi)^2\over 2\underline\sigma^2 t}} {1\over {1-e^{-{\pi^2\over \underline \sigma ^2t
}}}}\nonumber\\
&\to & 0, 
\end{eqnarray*}
since ${\rm Leb}(A_n)\to 0$ as $n\to\infty$. 
\end{proof}
The following lemma is vitally important. It is the strong Feller property in the classical case of linear probability space. But in the sublinear setting, it is not clear whether or not this holds in general. The proof of this result is quite involved where the regularity of $T_t$ (Lemma \ref{lem3.6}) plays an important role.
\begin{lem}\label{cf1}
Assume $\underline \sigma^2>0$ and $\varphi\in L_b({\cal B}(S^1))$, then for any $t>0$, $u(t,x)=T_t \varphi(x)$ given by (\ref{eqn3.7}) is $C^{1,2}$ and a classical solution of (\ref{eqn3.8}).
\end{lem}
\begin{proof}
Consider $\varphi\in L_b({\cal B}(S^1))$. First note there exists an increasing sequence of simple functions $\varphi_n^{(1)} \uparrow\varphi$ with $||\varphi_n^{(1)}||_\infty\leq ||\varphi||_\infty$. Thus by the monotone convergence of sublinear expectation
 we know that
$$u_n^{(1)}(t,x)=\E \varphi_n^{(1)}(x+B_t)\uparrow \E \varphi(x+B_t)=u(t,x).$$
Denote 
$$\varphi_n^{(1)}=\sum_{i=1}^{2^n} x_i I_{A_i^1},$$
where $\{A_i^1\}$ are Borel sets on $S^1$. By a standard result (c.f. Taylor \cite{taylor}), there exists a finite number of open intervals whose union is denoted by $B_i^0$ such that $A_i^1 \bigtriangleup B_i^0$ can be sufficiently small. Define
$$\varphi_n^{(2)}=\sum_{i=1}^{2^n} x_i I_{B_i^0}.$$
Then 
$$|\E\varphi_n^{(2)}(x+B_t)-\E \varphi_n^{(1)}(x+B_t)|\leq \sum_{i=1}^{2^n} |x_i| \E I_{A_i^1 \bigtriangleup B_i^0} (x+B_t).$$
As the Brownian motion is nondegenerate ($\underline\sigma^2>0$), so by Lemma \ref{lem3.6}, the expectation $\E I_{A_i^1 \bigtriangleup B_i^0} (x+B_t)$ can be sufficiently small since the Lebesgue measure of ${A_i^1 \bigtriangleup B_i^0}$ is sufficiently small. Thus $u_n^{(2)}(t,x)=\E \varphi_n^{(2)}(x+B_t)$ is sufficiently close to $u_n^{(1)}(t,x)$. 

Now note that one can find easily an increasing (or decreasing) sequence of continuous functions to approximate $I_{B_i^0}$. Thus there exists an increasing sequence of continuous functions $\varphi_{nm}^{(3)}\uparrow \varphi_{n}^{(2)}$ as $m\to \infty$ with $||\varphi_{nm}^{(3)}||_\infty\leq ||\varphi_n^{(2)}||_\infty$. By monotone convergence theorem, 
$$u_{nm}^{(3)}(t,x) =\E\varphi_{nm}^{(3)}(x+B_t)\uparrow u_{n}^{(2)}(t,x).$$
Summarizing above, we conclude there exists a sequence of continuous functions $\varphi_n$ 
such that 
$$u_{n}(t,x) =\E\varphi_{n}(x+B_t)\to u(t,x)=\E\varphi(x+B_t).$$

For any given $\delta>0$, by Krylov's result of the regularity of fully nonlinear parabolic partial differential equation of non-degenerate type (Krylov \cite{krylov1, krylov}), we know that
$$|D_t u_n(\delta, x)|+|D_x u_n(\delta, x)|\leq M,$$ 
for a constant $M>0$ being independent of $n$ and $x$. Thus the sequence $u_n(\delta, x) =(T_\delta \varphi_n)(x)$ of continuous functions is equi-continuous. Thus its limit  
$u(\delta, x) =(T_\delta \varphi)(x)$ is continuous in $x$. As 
$T_t\varphi=T_{t-\delta}T_\delta \varphi$, by Krylov's result again, we can see that $u(t,x)=T_t \varphi(x)$ given by (\ref{eqn3.7}) is $C^{1,2}$ in $(t,x)$ for any $t>0$.
\end{proof}

\begin{thm}\label{hz24}
Let $T_t$ be the Markovian semi-group defined by (\ref{eqn3.7}) with the $G$-Brownian motion
on the unit circle $S^1=[0,2\pi]$ with normal distribution $N(0,[\underline \sigma^2 t,\overline \sigma ^2 t])$, where
$\overline \sigma ^2   \geq \underline \sigma^2>0$ are constant. Then  
\begin{eqnarray}\label{cf17}
\tilde T\varphi={1\over {2\pi}}\int_0^{2\pi} (T_\delta \varphi)(x)dx,\ \varphi\in L_b({\cal B}(S^1)),\ \delta>0,
\end{eqnarray}
is independent of $\delta>0$ and is the unique invariant expectation of $T_t,\ t\geq 0$. Moreover, $T_t\varphi\to \tilde T\varphi$ as 
$t\to \infty$. 
\end{thm}
\begin{proof}
For each $\varphi\in L_b({\cal B}(S^1))$, define 
$m(\varphi) $ as integral of $\varphi$ with respect to the Lebesgue measure (normalised)
\begin{eqnarray}
m(\varphi)={1\over 2\pi}\int _{0}^{2\pi}\varphi(x)dx.
\end{eqnarray}
Set
\begin{eqnarray*}
{T}^{\overline\sigma}_t\varphi(x)=\int _{0}^{2\pi}p^{\overline\sigma}(t,x,y)\varphi(y)dy,
\end{eqnarray*}
and
\begin{eqnarray*}
T^{\underline\sigma}_t\varphi(x)=\int _{0}^{2\pi}p^{\underline\sigma}(t,x,y)\varphi(y)dy,
\end{eqnarray*}
where $p^{\overline\sigma}$ and $p^{\underline\sigma}$, the density of the transition probabilities of Brownian motions $\overline\sigma W_{\cdot}$ and $\underline\sigma W_{\cdot}$, respectively, are given by 
\begin{eqnarray}
p^{\overline\sigma}(t,x,y)=\sum\limits _{k\in {\mathbb Z}} {1\over \sqrt{2\pi\overline \sigma^2t}}{\rm e}^{-{(x-y-2k\pi)^2\over 2\overline \sigma ^2t}},
\end{eqnarray}
and
\begin{eqnarray}
p^{\underline\sigma}(t,x,y)=\sum\limits _{k\in {\mathbb Z}} {1\over \sqrt{2\pi\underline \sigma^2t}}{\rm e}^{-{(x-y-2k\pi)^2\over 2\underline \sigma ^2t}}.
\end{eqnarray}
Here $W_.$ is the classical Brownian motion on $S^1$. These standard Poisson summation formulae of heat kernels can be obtained using Fourier analysis or stochastic method (c.f. Elworthy \cite{david}) 
It is easy to see that if $\varphi$ is convex, then $T^{\overline\sigma}_t\varphi(x)$ is a convex function of $x$ for each $t$ and $T_t\varphi(x)=T^{\overline\sigma}_t\varphi(x)$. If $\varphi$ is concave, then $T_t\varphi(x)=T^{\underline\sigma}_t\varphi(x)$ which is a concave function of $x$ for each $t$.
Moreover, it is well-known that the Lebesgue measure is the invariant measure of Brownian motion on $S^1$ (c.f. Proposition 4.5 and Corollary of Theorem 4.6 in Chapter V in \cite{ikeda}), so
\begin{eqnarray*}
mT^{\underline\sigma}_t\varphi=m\varphi,\ \ mT^{\overline\sigma}_t\varphi=m\varphi, \ \ {\rm for } \ t\geq 0,
\end{eqnarray*}
and as $t\to\infty$, for any $x\in [0,2\pi]$
\begin{eqnarray*}
T^{\underline\sigma}_t\varphi(x)\to m\varphi,\ \ T^{\overline\sigma}_t\varphi(x)\to m\varphi.
\end{eqnarray*}
Thus if $\varphi$ is convex or concave,
\begin{eqnarray}\label{hz21}
mT_t\varphi=m\varphi,
\end{eqnarray}
and
as $t\to\infty$, for any $x\in [0,2\pi]$
\begin{eqnarray}\label{hz22}
T_t\varphi(x)\to m\varphi.
\end{eqnarray}

Now we consider $\varphi$ being a polynomial function defined in $[0,2\pi]$. It is well-known that there exist a convex function $\varphi_1$ and a concave function
$\varphi_2$ such that
$\varphi=\varphi_1+\varphi_2$. By the sublinearity of $T_t$, we have
\begin{eqnarray}
T_t\varphi_1(x)-T_t(-\varphi_2)(x)\leq T_t\varphi(x)\leq T_t\varphi_1(x)+T_t\varphi_2(x).
\end{eqnarray}
It follows from the linearity of $m$ that
\begin{eqnarray*}
mT_t\varphi\leq mT_t\varphi_1+mT_t\varphi_2=m\varphi_1+m\varphi_2=m(\varphi_1+\varphi_2)=m\varphi,
\end{eqnarray*}
and
\begin{eqnarray*}
mT_t\varphi\geq mT_t\varphi_1-mT_t(-\varphi_2)=m\varphi_1-m(-\varphi_2)=m(\varphi_1+\varphi_2)=m\varphi.
\end{eqnarray*}
So (\ref{hz21}) holds true for any polynomial function $\varphi$. It then follows from  an approximation argument using the Weierstrass theorem that (\ref{hz21}) is also true for $\varphi \in C([0,2\pi])$.

Moreover, for any polynomial function $\varphi$, as above $\varphi=\varphi_1+\varphi_2$,
$\varphi_1$ is  convex and $\varphi_2$ is concave, we have when
$t\to \infty$,
\begin{eqnarray*}
T_t\varphi_1(x)+T_t\varphi_2(x)\to m\varphi_1+m\varphi_2=m(\varphi_1+\varphi_2)=m\varphi,
\end{eqnarray*}
and
\begin{eqnarray*}
T_t\varphi_1(x)-T_t(-\varphi_2(x))\to m\varphi_1-m(-\varphi_2)=m(\varphi_1+\varphi_2)=m\varphi.
\end{eqnarray*}
Thus  (\ref{hz22}) holds for any polynomial $\varphi$.

Now we consider $\varphi \in C([0,2\pi])$. First note by the Weierstrass approximation theorem,  
for any $\epsilon >0$, there exists a polynomial $\tilde\varphi$
such that $\sup\limits _{x\in [0,2\pi]}|\tilde \varphi(x)-\varphi(x)|<{1\over 3}\epsilon$. So $|T_t\tilde \varphi(x)-T_t\varphi(x)|<{1\over 3}\epsilon$
for any $x,t$ and $|m\tilde \varphi(x)-m\varphi(x)|<{1\over 3}\epsilon$. On the other hand, for such $\tilde\varphi$,
there exists $R>0$ such that for any $t\geq R$,
$|T_t\tilde \varphi(x)-m\tilde \varphi|<{1\over 3}\epsilon$.
Thus for $t\geq R$,
\begin{eqnarray}
|T_t\varphi(x)-m \varphi |
\leq |T_t\varphi(x)-T_t\tilde \varphi(x)|+|T_t\tilde \varphi(x)-m\tilde \varphi|+|m\tilde \varphi-m\varphi |<\epsilon.
\end{eqnarray}
This leads to (\ref{hz22}) for any $\varphi \in C([0,2\pi])$.

Now consider $\varphi\in L_b({\cal B}(S^1))$. By Lemma \ref{cf1}, for any $\delta>0$,
$(T_\delta \varphi)(x)$ is continuous in $x$.
 Applying (\ref{hz22}) for continuous function, we have
$$T_t\varphi=T_{t-\delta}T_\delta \varphi\to m(T_\delta \varphi)=(mT_\delta )\varphi,\ {\rm as}\ t\to \infty.$$
So the last statement of the theorem is verified. 
But $T_t\varphi$ is independent of $\delta$, then $m(T_\delta \varphi)$ is independent of $\delta>0$, which means $m(T_{\delta_1})=m(T_{\delta_2})$ for any $\delta_1, \delta_2>0$. Define $\tilde T: L_b({\cal B}(S^1))\to \R^1$
$$\tilde T\varphi=(mT_\delta)\varphi,\ \delta>0.$$
Then for any $t\geq 0$,
$$\tilde TT_t\varphi=mT_\delta T_t\varphi=mT_{t+\delta}\varphi=\tilde T\varphi.$$
Thus $\tilde T$ is an invariant expectation. The uniqueness follows from the convergence of $T_t\varphi$.
\end{proof}
\begin{rmk} \label{zhao43}
(i) From the proof, we can see that when $\varphi\in C([0,2\pi])$, $\tilde T\varphi={1\over {2\pi}}\int_0^{2\pi} \varphi (x)dx$.

(ii)
We do not attempt to give the result in Theorem \ref{hz24} in a great generality e.g. of Brownian motions on a compact manifold. Here we only show such a result as an example. More general case will be treated in future publications.
\end{rmk}

As we have proved the invariant expectation $\tilde T$ of $G$-Brownian motion on $S^1$ exists, we can follow the procedure in Section \ref{zhao1751} to construct the canonical process $\hat X$ and the canonical dynamical system on the path space.

Applying Theorem \ref{hz20}, in the following 
we prove that the $G$-Brownian motion on the unit circe is ergodic. Firstly, we need the following proposition
where the no mean-uncertainty condition needed in Theorem \ref{hz20} is proved in (ii) below. Recall $\Omega^*=C(\R, \R^d)$ with the topology given in (\ref {eqn3.7a}).

\begin{prop}\label{prop3.26}
Consider the G-Brownian motion on the unit circle $S^1=[0,2\pi]$ with normal distribution $N(0, [\underline \sigma^2 t,\overline \sigma^2 t])$, where $\overline \sigma^2\geq\underline \sigma^2>0$. The following results hold:

(i) The stationary process $\hat X$ defined in (\ref{cf12}) has a continuous modification $\tilde X$ and is stochastically continuous. 

(ii) For each $\varphi\in L_b({\cal B}(S^1))$, $\varphi(\tilde X(0))$ has no mean-uncertainty with respect to the invariant expectation $\tilde {\cal E}$.

(iii) There exists a weakly compact family of probability measures ${\cal P}$ on $(\Omega^*,{\cal B}(\Omega^*))$ such that
\begin{eqnarray*}
\hat \E^{\tilde {T}}[\xi]=\sup\limits_{P\in {\cal P}}E_{P}[\xi], \ \ \xi\in Lip_{b, cyl}(\Omega^*).
\end{eqnarray*}

(iv) The invariant expectation $\tilde{T}$ is regular.

(v) 
Define for each $\xi\in {\cal B}(\Omega^*)$, the upper expectation 
\begin{eqnarray}\label{cf15}
\E^{*}[\xi]=\sup\limits _{P\in {\cal P}}E_P[\xi]. 
\end{eqnarray}
Then for any $F_n\in {\cal B} (\Omega^*)$ such that  $I_{F_n}\downarrow 0$, then $\E^{*} [I_{F_n}]\downarrow 0$.
Thus $\E^{*}$ is regular.
\end{prop}
\begin{proof} (i) 
Note by the sublinear expectation representation theorem, for the sublinear expectation $\E^{\tilde {T}}$ on $(\hat\Omega, 
L_0^1(\hat\Omega ))$, there exists a family of linear expectations $\{E_{\theta}: \theta\in \Theta\}$ such that
\begin{eqnarray}\label{cf16}
\E^{\tilde {T}}[X]=\sup\limits_{\theta\in \Theta}E_{\theta}[X], \ \ X\in L_0^1(\hat \Omega).
\end{eqnarray}
Note further that if $\{\varphi_n\}_{n=1}^{\infty}\subset C_{b,Lip}((S^1)^m)$ satisfies $\varphi_n\downarrow 0$, then by a similar argument
as in the proof of Lemma 3.3 of Chapter I in Peng \cite{peng2010}, 
\begin{eqnarray*}
\E^{\tilde {T}}[\varphi_n(\hat \omega_{t_1}, \hat \omega_{t_2},\dots, \hat \omega_{t_m})]\downarrow 0, \ \ {\rm as } \ n\to \infty,
\end{eqnarray*}
and it follows from (\ref{cf16}) that 
\begin{eqnarray*}
\E^{\tilde {T}}[\varphi_n(\hat \omega_{t_1}, \hat \omega_{t_2},\dots, \hat \omega_{t_m})]
=\sup\limits _{\theta\in \Theta}E_{\theta}[\varphi_n(\hat \omega_{t_1}, \hat \omega_{t_2},\dots, \hat \omega_{t_m})].
\end{eqnarray*}
But for each $\theta \in \Theta$, $E_{\theta}$ is controlled by $\E^{\tilde {T}}$. Thus $E_{\theta}[\varphi_n(\hat \omega_{t_1}, \hat \omega_{t_2},\dots, \hat \omega_{t_m})]\downarrow 0$ as $n\to \infty$. So by the Daniell-Stone
Theorem (c.f. Peng \cite{peng2010}), there is a unique probability measure $Q_{\{t_1,t_2,\dots,t_m\}}$ on $((S^1)^m,{\cal B}((S^1)^m))$ such that 
\begin{eqnarray*}
E_{\theta}[\varphi_n(\hat \omega_{t_1}, \hat \omega_{t_2},\dots, \hat \omega_{t_m})]=E_{Q_{\{t_1,t_2,\dots,t_m\}}}[\varphi_n(\hat \omega_{t_1}, \hat \omega_{t_2},\dots, \hat \omega_{t_m})].
\end{eqnarray*}
Denote $\T=\{\underbar t=\{t_1,t_2,\dots,t_m\}: t_1<t_2<\dots <t_m, m\in {\mathbb N}\}.$ Thus we have a family of finite
dimensional distributions $\{Q_{\underbar t}, \underbar t\in \T\}$.  It is easy to check that $\{Q_{\underbar t}, \underbar t
\in \T\}$ is consistent. By Kolmogorov's consistence theorem, there is a probability measure $Q$ on $(\hat \Omega,\hat 
{\F})$ such that $\{Q_{\underbar t}, \underbar t
\in \T\}$ is the finite dimensional distribution of $Q$. The probability distribution $Q$ is unique
as by Daniell-Stone theorem, its finite dimensional distribution is unique so the uniqueness of $Q$  follows
from the  monotone class theorem. It is now clear that $E_\theta[X]=E_{Q}[X]$ for any $X\in Lip_{b, cyl}(\hat \Omega)$.
Thus it follows from (\ref{cf16}) that 
\begin{eqnarray*}
\E^{\tilde {T}}[X]=\sup\limits_{Q\in {\cal P}_e}E_{Q}[X], \ \ X\in Lip_{b, cyl}(\hat \Omega),
\end{eqnarray*}
where ${\cal P}_e$ is a family of probability measures on $(\hat \Omega, {\cal B} (\hat \Omega))$.
Define the associated capacity:
$$\hat c(A):=\sup_{Q\in {\cal P}_e} Q (A), \ A\in {\cal B} (\hat \Omega),$$
and the upper expectation of each ${\cal B}(\hat \Omega)$-measurable real valued function $X$ which makes the following definition meaningful 
$$\hat\E^{\tilde {T}}[X]=\sup\limits_{Q\in {\cal P}_e}E_{Q}[X].$$
On the space $Lip_{b, cyl}(\hat \Omega)$, $\E^{\tilde {T}}=\hat\E^{\tilde {T}}$.
Consider the canonical process $\hat X$ on $(\hat \Omega, L_0^1(\hat \Omega), \E^{\tilde {T}}, \hat\theta _t)$. 
For  $t\geq s$, by conditional expectation, 
\begin{eqnarray}\label{cf14}
&&
\hat\E^{\tilde {T}}(\hat X(t)-\hat X(s))^4\nonumber\\
&=&\E^{\tilde {T}}(\hat X(t)-\hat X(s))^4\nonumber\\
&=&\E^{\tilde {T}}[\E^{\tilde {T}}[(\hat X(t)-\hat X(s))^4|{\cal F}_s]]\nonumber\\
&=&\E^{\tilde {T}}[T_{t-s}\varphi(y)|_{y=\hat X(s)}]\nonumber\\
&\leq& 
c|t-s|^2,
\end{eqnarray}
where $\varphi(y)=(y-\hat X(s))^4$,
$c>0$ is a constant independent of 
$t$ and $s$. Then by the Kolmogorov continuity theorem for sublinear expectations (Theorem 1.36, Chapter VI, Peng \cite{peng2010}), 
the processes $\hat X$ has a continuous modification, denoted by $\tilde X$ such that $\hat c({\tilde X}_t\neq \hat X_t)=0$.
Note for any $\delta>0$, 
\begin{eqnarray*}
\hat\E^{\tilde {T}}(\hat X(t)-\hat X(s))^4\geq \hat\E^{\tilde {T}}[(\hat X(t)-\hat X(s))^4 I_{\{|\hat X(t)-\hat X(s)|>\delta\}}]\geq \delta^4 \hat\E^{\tilde {T}}I_{\{|\hat X(t)-\hat X(s)|>\delta\}},
\end{eqnarray*}
so
\begin{eqnarray*}
 \hat\E^{\tilde {T}}I_{\{|\hat X(t)-\hat X(s)|>\delta\}}\leq \delta^{-4} \hat\E^{\tilde {T}}(\hat X(t)-\hat X(s))^4,
 \end{eqnarray*}
 thus the stochastic continuity follows from (\ref{cf14}).

(ii). Now we prove for any $\varphi\in L_b({\cal B}(S^1))$, $\varphi(\tilde X(0))$ has no mean-uncertainty. We follow the 3-step approximation procedure of using a sequence of continuous functions to approximate $\varphi$. Note the no mean-uncertainty of $\varphi(\tilde X(0))$ when $\varphi\in C_b(S^1)$ follows from (\ref{cf17}) and the fact that $\tilde T$ is a Lebesgue integral in this case automatically. Adopting the same notation as in the proof of Lemma \ref{cf1}, consider the increasing sequence of continuous functions $\varphi_{nm}^{(3)}\uparrow \varphi_n^{(2)}$, when $m\to\infty$. First note by Remark \ref{zhao43} (i),
\begin{eqnarray}\label{eqn3.20}
\tilde {\cal E}(-\varphi_{nm}^{(3)}(\tilde X(0)))=-\tilde {\cal E}(\varphi_{nm}^{(3)}(\tilde X(0))).
\end{eqnarray}
By Lemma \ref{lem2.10}, we have $\varphi_n^{(2)}(\tilde X(0))$ has no mean uncertainty,
\begin{eqnarray}\label{eqn3.23}
\tilde {\cal E}(-\varphi_{n}^{(2)}(\tilde X(0)))=-\tilde {\cal E}(\varphi_{n}^{(2)}(\tilde X(0))).
\end{eqnarray}
But 
\begin{eqnarray}
|\tilde {\cal E}(\varphi_{n}^{(2)}(\tilde X(0)))-\tilde {\cal E}(\varphi_{n}^{(1)}(\tilde X(0)))|\leq \sum_{i=1}^r |x_i|\tilde {\cal E} (I_{A_i^1 \bigtriangleup B_i^0}(\tilde X(0))),
\end{eqnarray}
and
\begin{eqnarray}
|\tilde {\cal E}(-\varphi_{n}^{(2)}(\tilde X(0)))-\tilde {\cal E}(-\varphi_{n}^{(1)}(\tilde X(0)))|\leq \sum_{i=1}^r |x_i|\tilde {\cal E} (I_{A_i^1 \bigtriangleup B_i^0}(\tilde X(0))),
\end{eqnarray}
so $\varphi_n^{(1)}(\tilde X(0))$ has no mean uncertainty. As $\varphi_{n}^{(1)}\uparrow \varphi$, by Lemma \ref{lem2.10} again, $\varphi(\tilde X(0))$ has no mean uncertainty,
$$\tilde {\cal E}(-\varphi(\tilde X(0)))=-\tilde {\cal E}(\varphi(\tilde X(0))).$$

(iii). In the following we will find a weakly compact family of probability measures  ${\cal P}$ on $(\Omega^*,{\cal B}(\Omega^*))$ such that the upper expectation (\ref{cf15}) gives a sublinear expectation on  ${\cal P}$ on $(\Omega^*,{\cal B}(\Omega^*))$ 
with finite dimensional expectation of $\varphi(\omega_{t_1}^*, \omega_{t_2}^*,\dots, \omega_{t_m}^*)$, $t_1<t_2<\dots<t_m$, to be ${T}_{t_1,t_2,\dots,t_m}^{\tilde {T}}\varphi$ for $\varphi\in L_b({\cal B}((S^1)^m))$.

For each $Q\in {\cal P}_e$, let $Q\circ \tilde X^{-1}$ which is a probability measure on 
$(\Omega^*,{\cal B}(\Omega^*))$ induced by $\tilde X$ from $Q$ and set ${\cal P}_1=\{Q\circ \tilde X^{-1}: Q\in {\cal P}_e\}$. 
Then similar to (\ref{cf14}), we have 
\begin{eqnarray*}
\hat\E^{\tilde {T}}(\tilde X(t)-\tilde  X(s))^4=\hat\E^{\tilde {T}}(\hat X(t)-\hat  X(s))^4
\leq  c|t-s|^2, \ \ t,s\in {\mathbb R}.
\end{eqnarray*}
Applying the moment criterion for the tightness of Kolmogorov-Chentsov's type, we conclude that ${\cal P}_1$
as a family of probability measures on $(\Omega^*,{\cal B}(\Omega^*))$ is tight. Denote ${\cal P}=\bar{\cal P}_1$ the closure of 
${\cal P}_1$ under the topology of weak convergence. Then ${\cal P}$ is weakly compact. Note 
\begin{eqnarray*}
\hat\E^{\tilde {T}}[\xi]=\sup\limits_{P\in {\cal P}_1}E_{P}[\xi], \ \ \xi\in Lip_{b, cyl}(\Omega^*).
\end{eqnarray*}
For each $\xi\in Lip_{b, cyl}(\Omega^*)$, from Lemma 3.3 of Chapter I in \cite{peng2010}, we get $ \hat\E^{\tilde {T}} [|\xi-(\xi\wedge N)\vee(-N)|]\downarrow 0$ as $N\to \infty$. So 
\begin{eqnarray*}
\hat\E^{\tilde {T}}[\xi]=\sup\limits_{P\in {\cal P}}E_{P}[\xi], \ \ \xi\in Lip_{b, cyl}(\Omega^*).
\end{eqnarray*}

(iv). 
 For any $A_n\in {\cal B} (S^1)$ such that $I_{A_n}\downarrow 0$, then by (\ref{cf17}) and Lemma \ref{lem3.6}, we have $\tilde{T} [I_{A_n}]\downarrow 0$, i.e. $\tilde{ T}$ is regular.

(v). For $\cal P$ given in (iii), we define the associated $G$-capacity
$$ c^*(F):=\sup_{P\in \cal P} P(F), \ F\in {\cal B}(\Omega^*),$$
and upper expectation for each ${\cal B}(\Omega^*)$-measurable real valued 
function $\xi$ which makes the following definition meaningful:
$$\E^{*}[\xi]:= \sup_{P\in \cal P} E_P[\xi].$$
On $Lip_{b, cyl}(\Omega^*)$, $\E^{*}=\hat\E^{\tilde {T}}$ and  as ${\cal P}$ 
is a weakly compact family of probability measures on $(\Omega^*, {\cal B}(\Omega^*))$, we have for any continuous $\xi_n$ and $\xi_n\downarrow 0$, $\E^{*}[\xi_n]\downarrow 0$ as $n\to \infty$.
Now consider for any $F_n\in {\cal B}(\Omega^*)$, such that $I_{F_n}\downarrow 0$. 
Define
\begin{eqnarray*}
C_n=\{\omega \in \Omega ^*: \rho(\omega, F_n)\leq {1\over n}\}, 
\ 
D_n=\{\omega \in \Omega ^*: \rho(\omega, F_n)< {2\over n}\}.
\end{eqnarray*}
Moreover, define
\begin{eqnarray*}
\xi_n(\omega)=n[{\rm min}
\{{\rho(\omega, D_n^c), \rho(C_n, D_n^c)}\}].
\end{eqnarray*}
Then it is easy to see that $\xi_n(\omega)$ is continuous in $\omega \in \Omega ^*$ and $I_{F_n}\leq \xi_n$. As when $\xi_n\downarrow 0$, we have that $\E^{*}[\xi_n]\downarrow 0$ as $n\to \infty$, it follows that $\E^{*} [I_{F_n}]\downarrow 0$.
\end{proof}
From the result of Proposition \ref{prop3.26} and Theorem \ref{prop3.6}, we can conclude that the canonical dynamical system generated by the semigroup of the $G$-Brownian motion on the unit circle is continuous.

\begin{thm}\label{thm3.27}
The invariant expectation of the $G$-Brownian motion on the unit circle
$S^1=[0,2\pi]$ with normal distribution $N(0,[\underline \sigma^2 t,\overline \sigma ^2 t])$, where
$\overline \sigma ^2   \geq \underline \sigma^2>0$ are constant, is ergodic.
\end{thm}
\begin{proof} 
Consider $\varphi\in L_b({\cal B}(S^1))$ with $T_t \varphi=\varphi$ and $T_t (-\varphi)=-\varphi, t\geq 0$. 
From the convergence result 
that as $t\to \infty$, $T_t\varphi\to \tilde T\varphi$ in Theorem \ref{hz24}, it is easy to know that $\varphi=\tilde T\varphi$ so $\varphi$ is constant. 
By Theorem \ref {hz20}, the invariant expectation is ergodic.
\end{proof}

\begin{rmk}
Following the regularity result of  $\E^{*}$ in Proposition \ref{prop3.26}, 
and the ergodicity results for the G-Brownian motion on the unit circle, it follows that SLLN holds by Theorem \ref{thm4.7}. 
\end{rmk}

Inspired by Theorem \ref{hz20},  we observe that the study of the ergodicity of the invariant expectation $\tilde {T}$ is equivalent to the study
of the spectrum of the semigroup ${T}_t$ on the space of $L_b({\cal B}(\R^d))$.  It is noted that due to the constant preserving property of the sublinear expectation, the
sublinear semigroup ${T}_t$ on $L_b({\cal B}(\R^d))$ has eigenvalue $1$. Theorem \ref{hz20} says that ergodicity is equivalent to $1$ being a simple eigenvalue of $T_t$ on $L_b({\cal B}(\R^d))$ as $|\varphi(X(0))|^2$ has no mean uncertainty.

Now we consider the relation of the eigenvalues of $T_t$ and its infinitesimal generator $\mathbb G$. 
First assume $1$ is a simple eigenvalue of $T_t$.
Recall ${\mathbb G}(u)={1\over 2}\overline\sigma ^2u_{xx}^+-{1\over 2}\underline\sigma ^2u_{xx}^-$ and $u(t,x)={T}_t\varphi(x)$ satisfying (\ref{eqn3.8}).
It is easy to see that ${\mathbb G}(c)=0$ for any constant $c$. This suggests that $0$ is an eigenvalue of the generator ${\mathbb G}$ in the space of twice differentiable functions.
However, if $\varphi$ is continuous and a viscosity solution of ${\mathbb G}(\varphi)=0$, it is easy to see that ${T}_t\varphi=\varphi$. So $\varphi$ is constant. This means $0$ is a simple eigenvalue of $\mathbb G$. Conversely, now assume $0$ is a simple eigenvalue of $\mathbb G$. Consider $\varphi$ as a continuous function satisfying $T_t\varphi=\varphi$. As $T_t\varphi$ is a solution of (\ref{eqn3.8}), so $\mathbb G(\varphi)=0$. Thus $\varphi$ is a constant due to the spectrum assumption of $\mathbb G$. This correspondence is also true for sublinear Markovian semigroups and their infinitesimal generators in general cases.


From the above discussions, our result shows that as the $G$-Brownian motion on the unit circle is ergodic, so $0$ is a simple eigenvalue of the corresponding infinitesimal generator 
${\mathbb G}(\cdot)$. In fact, we can prove this result analytically without referring to the result of ergodicity.

\begin{prop}\label{zhao2020a}
Let a continuous function $\varphi$ be a viscosity solution of
\begin{eqnarray}\label{hz28}
{1\over 2}\overline\sigma ^2\varphi_{xx}^+-{1\over 2}\underline\sigma ^2\varphi_{xx}^-
=0, \ \ x\in [0,2\pi], \ \varphi(0)=\varphi(2\pi).
\end{eqnarray}
If $\underline{\sigma} ^2>0$, then $\varphi$ is constant.
\end{prop}
\begin{proof}
Let $\psi$ be a $C^2$ function on $[0,2\pi]$ such that $\psi\geq \varphi$ and $\psi(x)=\varphi(x)$ at certain $x\in [0,2\pi]$ with $\psi^{\prime\prime}(x)\neq 0$. Then
${1\over 2}\overline\sigma ^2\psi^{\prime\prime}(x)^+-{1\over 2}\underline\sigma ^2\psi^{\prime\prime}(x)^-
\geq 0 $. It is then obvious that
\begin{eqnarray}\label{hz29}
\underline\sigma ^2\psi^{\prime\prime}(x)^-
\leq \overline\sigma ^2\psi^{\prime\prime}(x)^+.\end{eqnarray}
If $\psi^{\prime\prime}(x)<0$, then $\psi^{\prime\prime}(x)^->0$ and $\psi^{\prime\prime}(x)^+=0$. This contradicts with (\ref{hz29}).
Thus $\psi^{\prime\prime}(x)\geq 0$ so $\psi$ is locally a convex function near $x$.

Similarly, let $\tilde \psi$ be a $C^2$ function on $[0,2\pi]$ such that $\tilde \psi\leq \varphi$ and $\tilde \psi(x)=\varphi(x)$ at certain $x\in [0,2\pi]$ with $\tilde\psi^{\prime\prime}(x)\neq 0$. Then
${1\over 2}\overline\sigma ^2\tilde\psi^{\prime\prime}(x)^+-{1\over 2}\underline\sigma ^2\tilde\psi^{\prime\prime}(x)^-
\leq 0 $. It is then obvious that
\begin{eqnarray}\label{hz30}
\overline\sigma ^2\tilde\psi^{\prime\prime}(x)^+
\leq \underline\sigma ^2\tilde\psi^{\prime\prime}(x)^-.\end{eqnarray}
If $\tilde\psi^{\prime\prime}(x)>0$, then $\tilde\psi^{\prime\prime}(x)^+>0$ and $\tilde\psi^{\prime\prime}(x)^-=0$. This contradicts with (\ref{hz30}).
Thus $\tilde\psi^{\prime\prime}(x)\leq 0$ and $\tilde\psi$ is locally a concave function near $x$.

A function $\varphi$ that satisfies the above two properties must be a linear function. Now from the periodic boundary of $\varphi$, we conclude easily that $\varphi$ is a constant. 
\end{proof}

\begin{rmk}
The condition $\underline{\sigma} ^2>0$ is crucial for Proposition \ref{zhao2020a}. Otherwise, any smooth concave periodic function $\varphi$ with period $2\pi$ satisfies (\ref{hz28}) since $\varphi_{xx}^+=0$. In that case, Brownian motion (degenerate) on $S^1$ fails to be ergodic.  So Theorem \ref{thm3.27} can be stated as follows:
\end{rmk}
\begin{thm}
The invariant expectation of the $G$-Brownian motion on the unit circle
$S^1=[0,2\pi]$ with normal distribution $N(0,[\underline \sigma^2 t,\overline \sigma ^2 t])$, where
$\overline \sigma ^2   \geq \underline \sigma^2$ are constant, is ergodic if and only if $\underline \sigma^2>0$.
\end{thm}

\section*{Acknowledgements}
We are grateful to the anonymous referees for their constructive comments which helped us to improve the paper substantially. 
We would like to acknowledge the financial support of Royal Society Newton Fund grant NA15034 and an EPSRC grant (ref. EP/S005293/2).

\vskip20pt

\noindent
{\bf Appendix: Proofs of Theorem \ref{hz5} and Lemma \ref{zhao523}}
\vskip5pt

{\it Proof of Theorem \ref{hz5}}. (i)$\Rightarrow$(ii). Assume $B\in    {\cal F}$ and $   \E{\rm I}_{{   \theta}^{-1}B\Delta B}=0$. Define
\begin{eqnarray}
B_{\infty}=\bigcap\limits _{n=0}^{\infty}\bigcup \limits _{i=n}^{\infty} {   \theta}^{-i}B.
\end{eqnarray}
Then it is easy to see that
\begin{eqnarray*}
{   \theta}^{-1}B_{\infty}=\bigcap\limits _{n=0}^{\infty}\bigcup \limits _{i=n+1}^{\infty} {   \theta}^{-i}B=B_{\infty}.
\end{eqnarray*}
Thus $B_{\infty}$ is an invariant set. By the assumption, we have
\begin{eqnarray}\label{hz4}
   \E{\rm I}_{B_{\infty}}=0 \ \ {\rm or }\ \    \E{\rm I}_{B_{\infty}^c}=0.
\end{eqnarray}
Note for any $n\in N$
\begin{eqnarray*}\label{hz1}
{   \theta}^{-n}B\bigtriangleup B&\subset &\bigcup\limits _{i=0}^{n-1}({   \theta}^{-(i+1)}B\bigtriangleup {   \theta}^{-i} B)
\nonumber
\\
&=&\bigcup\limits _{i=0}^{n-1}{   \theta}^{-i}({   \theta}^{-1}B\bigtriangleup B).
\end{eqnarray*}
So by the monotonicity and subadditivity of $   \E$ and the expectation preserving property of $   \theta$,
\begin{eqnarray}\label{eqn2.7}
   \E{\rm I}_{{   \theta}^{-n}B\bigtriangleup B}&\leq &   \E{\rm I}_{\bigcup\limits _{i=0}^{n-1}{   \theta}^{-i}({   \theta}^{-1}B\bigtriangleup B)}\nonumber\\
&\leq&   \E\left [\sum _{i=0}^{n-1}{\rm I}_{{   \theta}^{-i}({   \theta}^{-1}B\bigtriangleup B)}\right ]\nonumber\\
&\leq &\sum _{i=0}^{n-1}   \E{\rm I}_{{   \theta}^{-i}({   \theta}^{-1}B\bigtriangleup B)}\nonumber\\
&=&\sum _{i=0}^{n-1}   \E{\rm I}_{{   \theta}^{-1}B\bigtriangleup B}\nonumber\\
&=&0.
\end{eqnarray}
Moreover
\begin{eqnarray}\label{eqn2.8}
(\bigcup _{i=1}^{\infty}{   \theta}^{-i}B)\bigtriangleup B\subset \bigcup _{i=1}^{\infty}({   \theta}^{-i}B\bigtriangleup B).
\end{eqnarray}
Thus it follows from (\ref{eqn2.7}) and (\ref{eqn2.8}) that
\begin{eqnarray*}
   \E{\rm I}_{(\bigcup _{i=n}^{\infty}{   \theta}^{-i}B)\bigtriangleup B}&\leq&   \E{\rm I}_{\bigcup\limits _{i=0}^{\infty}({   \theta}^{-i}B\bigtriangleup B)}\\
&\leq&\sum _{i=0}^{\infty}   \E{\rm I}_{{   \theta}^{-i}B\bigtriangleup B}\\
&=&0.
\end{eqnarray*}
From the above we have
\begin{eqnarray}\label{hz2}
   \E{\rm I}_{(\bigcup _{i=n}^{\infty}{   \theta}^{-i}B)\setminus B}=0,
\end{eqnarray}
and
\begin{eqnarray}\label{hz3}
   \E{\rm I}_{B\setminus \bigcup _{i=n}^{\infty}{   \theta}^{-i}B}=0.
\end{eqnarray}
But note as $n\to\infty$,
\begin{eqnarray*}
{\rm I}_{(B\setminus \bigcup _{i=n}^{\infty}{   \theta}^{-i}B)}\uparrow {\rm I}_{(B\setminus\bigcap_{n=1}^{\infty}\bigcup _{i=n}^{\infty}{   \theta}^{-i}B)}=I_{B\setminus B_\infty},
\end{eqnarray*}
So by 
the monotone (increasing) convergence of sublinear expectation (\cite{peng2005}, \cite{denis}), we have as $n\to+\infty$,
\begin{eqnarray*}
   \E{\rm I}_{B\setminus \bigcup _{i=n}^{\infty}{   \theta}^{-i}B}\to   \E{\rm I}_{B\setminus B_{\infty}}.
\end{eqnarray*}
Thus it follows from (\ref{hz3}) that
\begin{eqnarray}
   \E{\rm I}_{B\setminus B_{\infty}}=0.
\end{eqnarray}
Moreover
\begin{eqnarray*}
{\rm I}_{(\bigcup _{i=n}^{\infty}{   \theta}^{-i}B)\setminus B}\downarrow {\rm I}_{B_{\infty}\setminus B}.
\end{eqnarray*}
It then follows by applying the monotonicity of sublinear expectation and (\ref{hz2}) that
\begin{eqnarray*}
   \E{\rm I}_{B_{\infty}\setminus B}=0.
\end{eqnarray*}
Note the regularity condition is not needed here. 
Thus
\begin{eqnarray*}
   \E{\rm I}_{B_{\infty}\bigtriangleup  B}=0.
\end{eqnarray*}
Now recall (\ref{hz4}). Consider the case that $   \E{\rm I}_{B_{\infty}}=0$. Note
\begin{eqnarray*}
0=   \E{\rm I}_{B\setminus B_{\infty}}&=&   \E{\rm I}_{B\setminus (B\cap B_{\infty})}\\
&=&   \E[{\rm I}_{B}-{\rm I}_{(B\cap B_{\infty})}]\\
&\geq &   \E[{\rm I}_{B}]-   \E[{\rm I}_{(B\cap B_{\infty})}]\\
&\geq &   \E[{\rm I}_{B}]-   \E[{\rm I}_{B_{\infty}}]\\
&= &   \E[{\rm I}_{B}].
\end{eqnarray*}
Hence
\begin{eqnarray*}
   \E[{\rm I}_{B}]=0.
\end{eqnarray*}
Now consider the case that $   \E{\rm I}_{B_{\infty}^c}=0$. Note
\begin{eqnarray*}
0=   \E{\rm I}_{B_{\infty}\setminus B}&=&   \E{\rm I}_{B^c\setminus (B^c\cap B_{\infty}^c)}\\
&=&   \E[{\rm I}_{B^c}-{\rm I}_{B^c\cap B_{\infty}^c}]\\
&\geq &   \E[{\rm I}_{B^c}]-   \E[{\rm I}_{B^c\cap B_{\infty}^c}]\\
&\geq &   \E[{\rm I}_{B^c}]-   \E[{\rm I}_{B_{\infty}^c}]\\
&=&   \E[{\rm I}_{B^c}].
\end{eqnarray*}
Thus
\begin{eqnarray*}
   \E[{\rm I}_{B^c}]=0.
\end{eqnarray*}
Therefore the assertion (ii) is proved.

(iii)$\Rightarrow$(iv).
Let $   \E{\rm I}_{A}>0$ and $   \E{\rm I}_{B}>0$. From (iii), we know that $   \E{\rm I}_{(\bigcup_{n=1}^{\infty}{   \theta}^{-n}A)^c}=0$.
It then follows together with applying subadditivity and monotonicity of $   \E$ that,
\begin{eqnarray*}
0<   \E{\rm I}_{B}&=&    \E[{\rm I}_{B\bigcap(\bigcup_{n=1}^{\infty}{   \theta}^{-n}A)}+{\rm I}_{B\bigcap(\bigcup_{n=1}^{\infty}{   \theta}^{-n}A)^c}]\\
&\leq &     \E[{\rm I}_{B\bigcap(\bigcup_{n=1}^{\infty}{   \theta}^{-n}A)}]+   \E[{\rm I}_{B\bigcap(\bigcup_{n=1}^{\infty}{   \theta}^{-n}A)^c}]\\
&\leq &     \E[{\rm I}_{\bigcup_{n=1}^{\infty}(B\bigcap{   \theta}^{-n}A)}]+   \E[{\rm I}_{(\bigcup_{n=1}^{\infty}{   \theta}^{-n}A)^c}]\\
&=&    \E[{\rm I}_{\bigcup_{n=1}^{\infty}(B\bigcap{   \theta}^{-n}A)}]\\
&\leq &\sum _{n=1}^{\infty}   \E[{\rm I}_{B\bigcap{   \theta}^{-n}A}].
\end{eqnarray*}
Thus it is obvious that there must exist $n\in {\mathbb N}$ such that
$   \E[{\rm I}_{(B\bigcap{   \theta}^{-n}A)}]>0$. So (iv) is proved.

(iv)$\Rightarrow$(i). Suppose that $B\in {\cal F}$ and ${  \theta }^{-1}B=B$.
If $   \E{\rm I}_B>0$ and $   \E{\rm I}_{B^c}>0$, then by assumption (iv) and invariant assumption of $B$,
\begin{eqnarray*}
0<   \E[{\rm I}_{B^c\bigcap{   \theta}^{-n}B}]=  \E[{\rm I}_{B^c\bigcap B}]=0.
\end{eqnarray*}
This is a contradiction and thus $   \E{\rm I}_B=0$ or $   \E{\rm I}_{B^c}=0$. So (i) is proved.

(ii)$\Rightarrow$(iii)under the regularity assumption. Assume $A\in   {\cal F}$ and $   \E{\rm I}_A>0$. Set
\begin{eqnarray*}
A_1=\bigcup_{n=1}^{\infty}{   \theta}^{-n}A.
\end{eqnarray*}
It is easy to see that ${   \theta}^{-1}A_1\subset A_1$ and
${   \theta}^{-n}A_1=\bigcup_{i=n+1}^{\infty}{   \theta}^{-i}A$. So $\{{   \theta}^{-n}A_1\}_{n\in {\mathbb N}}$ form a decreasing sequence of sets with limit
\begin{eqnarray}
{   \theta}^{-n}A_1\downarrow A_{\infty}=\limsup_{n}({   \theta}^{-n}A),
\end{eqnarray}
where the notation $A_{\infty}$ is used in the same fashion as in the proof of ``(i)$\Rightarrow$(ii)".
It is easy to see that
\begin{eqnarray*}
{   \theta}^{-1}A_{\infty}=A_{\infty}.
\end{eqnarray*}
Thus
\begin{eqnarray*}
   \E{\rm I}_{{   \theta}^{-1}A_{\infty}\bigtriangleup A_{\infty}}=0.
\end{eqnarray*}
According to assumption (ii), we know either  $   \E{\rm I}_{A_{\infty}}=0$ or
$   \E{\rm I}_{A_{\infty}^c}=0$. We claim the case that $   \E{\rm I}_{A_{\infty}}=0$ is impossible.
Otherwise, ${\rm I}_{A_{\infty}}=0$ quasi-surely. It then follows that  ${\rm I}_{{   \theta}^{-n}A_1}\downarrow
{\rm I}_{A_{\infty}}=0$ quasi-surely. So as $   \E$ is regular so that $   \E{\rm I}_{{   \theta}^{-n}A_1}\to 0$ as $n\to \infty$. However by the expectation
preserving property of $   \theta$, the definition of $A_1$ and the monotonicity of $   \E$,
\begin{eqnarray*}
   \E{\rm I}_{{   \theta}^{-n}A_1}=   \E{\rm I}_{A_1}\geq   \E{\rm I}_{{   \theta}^{-1}A}=   \E{\rm I}_{A}>0.
\end{eqnarray*}
We have a contraction. Thus $   \E{\rm I}_{A_{\infty}^c}=0$ holds. Then it follows that  $   \E{\rm I}_{A_{1}^c}=0$ as $A_{\infty}\subset
A_1$ so (iii) is proved. It is then obvious that all the four statements are equivalent under the regularity condition.
\hfill$\Box$
\vskip10pt

\noindent
{\it Proof of Lemma \ref{zhao523}}. Recall $S_n$ is defined by (\ref{eqn2.21}).
Let
$$\bar \xi=\limsup_{n\to \infty} {{S_n}\over n},$$
$ \epsilon>0, $ and
$$D=\{   \omega: \bar \xi(   \omega)>\bar \xi^*(  \omega)+\epsilon\}.$$
Our goal is to prove $   \E[-{\rm I}_D]=0$. Note $\bar \xi(   \theta   \omega)=\bar \xi(   \omega)$, and $
\bar \xi^*(  \theta  \omega)=\bar \xi^*(  \omega)$ quasi-surely,  so $D\in {\cal I}$.

Define
\begin{eqnarray*}
\xi^*(   \omega)&=&(\xi(   \omega)-\bar \xi^*(  \omega)-\epsilon){\rm I}_D(   \omega)\\
S_n^*(   \omega)&=&\xi^*(   \omega)+\dots+\xi^*(\theta_{n-1}^*   \omega)\\
M_n^*(   \omega)&=&\sup \{0, S_1^*(   \omega),\dots, S_n^*(   \omega)\}\\
F_n&=&\{   \omega: M_n^*(   \omega)>0\}
\end{eqnarray*}
and
$$F=\cup_nF_n=\{   \omega:\sup_{k\geq 1} {S_k^*\over k}>0\}.$$
Since $\xi^*(   \omega)=(\xi(   \omega)-\bar \xi^*
(  \omega)-\epsilon){\rm I}_D(   \omega)$ and $D=\{   \omega: \limsup_{k\to \infty}{{S_k}\over k}>\bar \xi^*+\epsilon\}$, it follows that $F=D.$ In fact, if $   \omega\in D$, then
$\sup_{k\geq 1} {{S_k}\over k}>\bar \xi^*+\epsilon,$
and by definition of $\xi^*$,
$ {S_k^*\over k}= {S_k \over k}-\epsilon-\bar \xi^*.$
So
$\sup_{k\geq 1}{S_k^*\over k}>0,$
i.e. $   \omega\in F.$ Therefore $D\subset F$. If $   \omega\notin D$, then $\xi^*(   \omega)=0$. Note $D\in {\cal I}$, so $\xi^*(  \theta_k   \omega)=0$, quasi-surely for all $k$. Therefore $S^*_k(   \omega)=0$ for all $k$, so $   \omega\notin F$. This tells us that $F\subset D$. Thus $F=D$. 

 Now applying the maximal ergodic theorem, we know that
 $   \E[\xi^* {\rm I}_{F_n}]\geq 0.$
 But
 \begin{eqnarray*}
   \E[\xi^* {\rm I}_{F_n}]&=&   \E[(\xi^*)^+ {\rm I}_{F_n}-(\xi^*)^- {\rm I}_{F_n}]\\
&\leq&     \E[(\xi^*)^+ {\rm I}_{F}-(\xi^*)^- {\rm I}_{F}+(\xi^*)^- {\rm I}_{F\setminus F_n}]\\
&\leq &    \E[\xi^* {\rm I}_{F}]+   \E[(\xi^*)^- {\rm I}_{F\setminus F_n}].
 \end{eqnarray*}
But $   \E [ (\xi^*)^- {\rm I}_{F\setminus F_n}]\downarrow 0$ as $n\to \infty$ because ${\rm I}_{F\setminus F_n}\downarrow 0$ and $   \E$ is regular. Thus
 $$   \E[\xi^* {\rm I}_{F}]\geq 0.$$
However, it follows that
\begin{eqnarray*}
0\leq    \E[(\bar\xi-\bar \xi^*-\epsilon){\rm I}_D]&\leq&    \E[(\bar\xi-\bar \xi^*){\rm I}_D]+   \E[-\epsilon {\rm I}_D]\\
&=&
\sup_{P\in {\cal P}}E_P[(\bar\xi-\bar \xi^*){\rm I}_D]+   \E[-\epsilon {\rm I}_D]\\
&=&
\sup_{P\in {\cal P}}E_P[E_P[(\bar\xi-\bar \xi^*){\rm I}_D|{\cal I}]]+   \E[-\epsilon {\rm I}_D]\\
&=&
\sup_{P\in {\cal P}}E_P[E_P[(\bar\xi-\bar \xi^*)|{\cal I}]{\rm I}_D]+   \E[-\epsilon {\rm I}_D]\\
&=&
\sup_{P\in {\cal P}}E_P[E_P[\bar\xi|{\cal I}]-\bar \xi^*]{\rm I}_D]+\epsilon    \E[-{\rm I}_D]\\
&\leq & \epsilon   \E[- {\rm I}_D].
\end{eqnarray*}
Thus $   \E[- {\rm I}_D]\geq 0$. On the other hand, $   \E[- {\rm I}_D]\leq 0$.
So
$   \E[- {\rm I}_D]= 0$ which equivalent to $v(D)=0$. Thus we get (\ref{hz11}). Define
$$\tilde D=\{   \omega: -\liminf_{n\to\infty}{S_n\over n}>-\underline \xi^*+\epsilon\}.$$
Applying the above result to $-\xi$, we can get $v(\tilde D)=0$. Therefore (\ref{hz13}) holds. \hfill $\Box$

   \end{document}